\documentclass[10pt,reqno]{amsart}
\usepackage[english]{babel}
\usepackage{a4wide,color,graphicx,amsmath,amssymb,verbatim,mathrsfs,latexsym, dsfont}
\usepackage[colorlinks=true]{hyperref}
\allowdisplaybreaks
\newtheorem{thm}{Theorem}
\newtheorem{prop}{Proposition}
\newtheorem{lem}{Lemma}

\newcommand{\R}{\mathbb{R}}	%
\newcommand{\N}{\mathbb{N}}	%
\newcommand{\eps}{\varepsilon}	%
\newcommand{\pa}{\partial}		%
\newcommand{\Div}{\textrm{div}\,}	%

\newcommand{\na}{\nabla}		%

\newcommand{\chf}[1]{{\raisebox{3pt}{\Large $\chi$}}_{#1}}
\newcommand{\IRd}{\int_{\R^d}}

\newcommand{\rhoe}{\rho_{\eps}}
\newcommand{\rhod}{\rho_{\delta}}

\newcommand{\Rd}{\mathbb{R}^d}
\newcommand{\dx}{\,{\rm d}x}
\newcommand{\dt}{\,{\rm d}t}
\newcommand{\rhoin}{\rho_{{\rm in}}}
\title[A quasilinear Keller-Segel with saturated discontinuous advection]{A quasilinear Keller-Segel model with saturated discontinuous advection}
\author{Maria Pia Gualdani}

\address{Department of Mathematics, The University of Texas at Austin, Austin TX, USA}
\email{gualdani@math.utexas.edu}
\author{Mikel Ispizua}
\address{Department of Mathematics, Autonoma University, Madrid, Spain.}
\email{mikel.ispizua@gmail.com}
\author{Nicola Zamponi}
\address{University\"{a}t Augsburg, Institut f\"ur Mathematik,  Augsburg, Germany}
\email{nicola.zamponi@alumni.uni-ulm.de}

\date{\today}

\begin{document}

	\thanks{M.P.G. is partially supported by NSF Grant DMS-2206677 and DMS-2511625. The authors would like to thank Franca Hoffmann and Inwon Kim for the fruitful discussions.}  
	
	\begin{abstract} 
		We consider the singular limit of a chemotaxis model of bacterial collective motion recently introduced in \cite{CalvezHoff2020}. The equation models aggregation-diffusion phenomena with advection that is discontinuous and depends sharply on the gradient of the density itself. The quasi-linearity of the problem poses major challenges in the construction of the solution and complications arise in the proof of regularity. Our method overcomes these obstacle by relying solely on entropy inequalities and the theory of monotone operators. We provide existence, uniqueness and $L^\infty$-smoothing estimates in any dimensional space.

	\end{abstract}

	\maketitle
\tableofcontents	
\section{Introduction and main result}
Since the classical experiment of J. Adler on \emph{Escherichia coli} \cite{Adler66} many experiments have confirmed the formation and propagation of traveling pulses of bacteria. The Keller-Segel model has been widely used as a basis to explain this phenomena in bacterial colonies driven by chemotaxis, see \cite{KellerSegel71,LiWang2021}: the model describes how the cell density moves towards a nutrient $N$ and a chemical molecule $S$ (\emph{chemoatractant}) which is generated by the cells themselves. The velocity of this movement is proportional to the gradient of these two quantities. The Keller-Segel model has been the subject of intense mathematical investigation.

Recent works (\cite{calvezPerthame2010, Perthame07, HillerOthmer00, FilberLaurencotPertham05, CMPS04, ErbanOthmer07}) have introduced a novel class of equations, originating from a mesoscopic description of the \emph{run-and-tumble} phenomenon at the individual cell level. Unlike the classical Keller-Segel framework, these models incorporate a saturated cell velocity (or flux), which is aligned with the gradients of chemical signals, specifically proportional to $\frac{\nabla S}{|\nabla S|}$ and $\frac{\nabla N}{|\nabla N|}$. The general form of this new class of equations is given by:
\begin{align*}
\partial_t \rho &= D_\rho \Delta \rho - \textrm{div}(\rho u_S + \rho u_N) , \\
\partial_t  S &= D_S \Delta S - \alpha S + \beta \rho, \\
\partial_t  N &= D_N \Delta N - \gamma \rho N,
\end{align*}
where 
\begin{align*}
u_S &= \chi_S J(\partial_t  S, |\nabla S|) \frac{\nabla S}{|\nabla S|},\quad u_N = \chi_N J(\partial_t  N, |\nabla N|) \frac{\nabla N}{|\nabla N|}.
\end{align*}

In \cite{calvezPerthame2010}, the authors derive an explicit expression for traveling pulses and prove the existence of steady states, corresponding to the physical clustering of bacteria (e.g., in nutrient-depleted conditions). Their numerical simulations show strong agreement with experimental data. The use of bounded fluxes, in contrast to Keller-Segel, provides key advantages for proving global well-posedness and for analyzing traveling pulses and long-time dynamics. Overall, this model more accurately captures the experimentally observed traveling bands and offers improvements  in describing bacterial movement.

More recently, Calvez and Hoffmann \cite{CalvezHoff2020} considered a simplified version  (absence of nutrient and quasi-stationary chemoattractant signal)
\begin{align}\label{CH20}
\left\{\begin{array}{rcll}
\pa_t\rho - \Delta\rho + \chi\Div\left(\rho \frac{\na S}{|\na S|} \right) &=& 0,\\ %
\vspace{2mm}
-\Delta S+\alpha S&=&\rho, %
\end{array}
\right.
\end{align}
and perform nonlinear stability analysis for the one dimensional case. In particular, they show that if the initial data is near, in some norm, to the equilibrium function $ \rho_\infty = \frac{\chi}{2}e^{-\chi|x-x_0|}$, then the corresponding solution converges to $ \rho_\infty$ exponentially fast. The constant $\alpha \ge 0$ denotes the natural decay of the chemical component and $\chi$ the chemoattractant sensitivity. {{A kinetic formulation and related stability analysis of (\ref{CH20}) in one dimensional space has been considered in  \cite{MW17, EY23, EY24, CFH24}. 

In this manuscript we analyze a model obtained from (\ref{CH20}) via the singular limit in $\alpha$. After rescaling $ S \mapsto S/\alpha$, the limit $\alpha\rightarrow\infty$ yields the following {\em{quasi-linear parabolic}} problem: 

\begin{equation}\label{original.eq}
		\pa_t\rho - \Delta\rho + \chi\Div\left(\rho \frac{\na\rho}{|\na\rho|} \right) = 0.%
\end{equation}

Unlike (\ref{CH20}),  the advection  in (\ref{original.eq}) depends on the gradient of the density, making it a competing term with the diffusion. Note in fact that 
$$
\Delta\rho  -  \chi\Div\left(\rho \frac{\na\rho}{|\na\rho|} \right) = \Div \left( \left( 1 - \frac{\chi \rho}{|\nabla \rho|}\right) \nabla \rho \right), 
$$
which shows a negative diffusion for profiles that have $\chi \rho \ge  |\nabla \rho|$. For example, Gaussian  functions $e^{-|x|^2}$  will make the diffusion coefficient negative around $|x|=0$.   Problem (\ref{original.eq}) is diffusive as long as  $\lvert\nabla\rho\rvert>\chi\rho$, and the steady-state is governed by the equation 
\begin{equation}\label{steadystate.eq}
\lvert\nabla\rho_\infty\rvert = \chi\rho_\infty.
\end{equation}
{
One solution to \eqref{steadystate.eq} is
\begin{equation}\label{steadystate}
\rho_\infty(x) = c e^{-\chi |x-x_0|},
\end{equation} 
for some $c$ and $x_0$ determined by the conservation of mass and property of invariance by translation of the equation. Other admissible (i.e.~belonging to $L^1(\R^d)$) solutions to \eqref{steadystate.eq} exist;
for example, for $d=1$ another solution is the ``multi-peak profile''
\begin{equation}\label{steadystate.2}
\rho_\infty(x) = \max \{ c_0\exp(-\chi|x-x_0|), c_1\exp(-\chi|x-x_1|) , ... c_k\exp(-\chi|x-x_k|) \}, 
\end{equation}
for a suitable choice of $k\geq 0$, $c_0 \ldots c_{k}>0$, and $x_0 \ldots x_{k+1}\in\R$, %
while for $d\geq 2$ other possible solution to (\ref{steadystate.eq})  is the ``factorized profile''
\begin{equation}\label{steadystate.3}
\rho_\infty(x) = c\exp\left( -\frac{\chi}{\sqrt{d}}\sum_{i=1}^d |x_i - x_{0,i}| \right) ,
\end{equation}
which equals the product of (suitably rescaled) single variable steady states.
}
The proof of existence of solution to (\ref{original.eq}) is still out of reach. The main difficulty is the control on the diffusion coefficient, which can change sign along the flow and destroy the parabolic nature of the equation. We believe that if the initial data is such that $\left(1-\frac{\chi\rho_{in}}{\lvert\nabla\rho_{in}\rvert}\right)$ is positive, the quantity $\left(1-\frac{\chi\rho}{\lvert\nabla\rho\rvert}\right)$  will remain non-negative throughout the evolution. This is, however, still a claim. Problem (\ref{original.eq}) has a more tractable form in one dimension 
$$
\pa_t\rho - \partial_{xx}\rho + \chi \partial_x\left(\rho  \;\textrm{sign}(\rho_x) \right) = 0, 
$$
and is the subject of current investigation \cite{GKIZ}.

As a first attempt to understand (\ref{original.eq}), we introduce an alternative model that retains some of the essential characteristics of (\ref{original.eq}) while being mathematically more tractable.  Specifically,  we consider
\begin{equation}
	\label{1}
	\left\{\begin{array}{rcll}
		\pa_t\rho - \Div\left(\left(1-\frac{\chi\rho}{\lvert\nabla\rho\rvert}\right)_+\nabla\rho\right)&=&0, &\quad\mbox{in }\Rd\times (0,\infty),\vspace{2mm}\\
		\vspace{2mm}
		\rho(0) &=& \rho_{in}, &\quad\mbox{in }\Rd,
	\end{array}
	\right.
\end{equation}
where $(\cdot)_+$ denotes $\max\{\cdot, 0 \}$.  Equation (\ref{1}) is a quasilinear differential equations. While this nonlinearity still poses major challenges in the construction of solutions, let's explain why (\ref{1}) is more within-reach than (\ref{original.eq}). The most important property of  (\ref{1}) is that the operator 
$$
\left(1-\frac{\chi}{|w|}\right)_+ w, \quad w \in \mathbb{R}^d,
$$
is monotone (for any constant $\chi >0$). More precisely, 
\begin{align}\label{monotonicity_prop}
\left[\left(1-\frac{\chi}{|w|}\right)_+ w  - \left(1-\frac{\chi}{|z|}\right)_+z\right]  \cdot ( w-z) \;  {{\ge}}\;  0, \quad \quad \forall \;  w,z\in \mathbb{R}^d. 
\end{align}
Note that the monotonicity is lost if we remove the positive part. 
This monotonicity  will  prove fundamental in two steps:  (i) in the construction of  solutions to a viscous version of (\ref{1}), and, most importantly,  (ii) in the vanishing viscosity limit.  The crux of the analysis will be the vanishing viscosity limit $\varepsilon \to 0$ and the corresponding identification of the limit of the nonlinear term 
$$
\left(1-\frac{\chi}{\lvert\nabla \log \rhoe\rvert}\right)_+\nabla \rhoe. 
$$
We will expand on this point after we summarize our main result: 
\begin{thm}%
\label{thr.ex}
Let $\rho_{{\rm in}}\in L^p(\Rd)\cap L^1_2(\Rd)$, $d\ge 1$, with $p>d$, $\rho_{{\rm in}} > 0$ a.e.~in $\Rd$  and $\int_{\Rd}\rho_{{\rm in}}|\log{\rho_{{\rm in}}}|<\infty$. For every $T>0$ there exists an unique solution $\rho:[0,T)\times\Rd\rightarrow \R$ to (\ref{1}) such that
\begin{align}\label{weak.formulation.1}
	\int_0^T\int_{\Rd}\partial_t\rho\,\varphi\,\dx \dt=-\int_0^T\int_{\Rd}\left(1-\frac{\chi \rho}{\lvert\nabla\rho\rvert}\right)_+\nabla\rho\nabla\varphi\,\dx \dt\quad\mbox{for } \;\;\forall\varphi\in L^2(0,T;H^1(\Rd)),
\end{align}
with $\rho \to \rho_{{\rm in}}$ strongly in $L^p(\Rd)$ as $t \to 0$. This solution belongs to $C(0,T,L^p(\mathbb{R}^d) \cap L^p(t,T, C^\alpha(\mathbb{R}^d))$ for all $t>0$ and has the following properties:
\begin{enumerate}
	\item[{\it (i)}] Mass conservation:
	\begin{equation*}
		\|\rho(t)\|_{L^1(\Rd)}=\|\rho_{{\rm in}}\|_{L^1(\Rd)}.
	\end{equation*}
	\item[{\it (ii)}] $L^p$ norm bound:
	\begin{equation*}
		\|\rho(t)\|_{L^p(\Rd)}\leq\|\rho_{{\rm in}}\|_{L^p(\Rd)}.
	\end{equation*}
	\item[{\it (iii)}] Moment evolution:
	\begin{equation*}
		\int_{\Rd}\rho(t)(1+|x|^2) \dx\leq e^{C\;t}\int_{\Rd}\rho_{{\rm in}}(1+|x|^2) \dx.
	\end{equation*}
	\item[{\it (iv)}] Entropy bound:
	\begin{align*}
		(\int_{\Rd }\rho|\log\rho|\dx)(T)+\int_0^T\int_{\Rd}\frac{\lvert\nabla\rhoe\rvert^2}{\rhoe}\dx\dt\leq \bar{C}[\rho_{{\rm in}}].
	\end{align*}
\item[{\it (iv)}] $L^{\infty}$ smoothing:
\begin{align*}
	 \sup_{t>0} \|\rho (t) \|_{L^{\infty}(\Rd)}\leq C\|\rho_{{\rm in}}\|_{L^p(\Rd)}\left(1+\frac{1}{t^{\frac{d}{2p}}}\right).
\end{align*}
\item[{\it (v)}] $L^{p}$ bound of the gradient:
\begin{align*}
	\|\nabla\rho\|_{L^p((t,T)\times \Rd)}< C(\rho_{{\rm in}}),\quad \forall \;  t>0.
\end{align*}
\end{enumerate}
\end{thm}

A common strategy for proving existence in nonlinear problems involves initially linearizing the equation and then applying fixed-point or contraction mapping arguments. If we were to follow this approach, we would begin by linearizing the diffusion term in (\ref{weak.formulation.1}) as
$$
\left(1-\chi \frac{z}{\lvert \nabla z \rvert}\right)_+\nabla \rho.
$$
However, because this linearization includes the term $|\nabla z|$, the functional framework for any fixed-point or contraction method would necessarily involve the Sobolev space $H^1$. The argument would also require regularity estimates in a higher space, typically $H^2$. This poses a problem, as we do not expect solutions to be in $H^2$. In fact, at points where $|\nabla \rho| = 0$, we anticipate a loss of smoothness. Some steady states of type (\ref{steadystate}) are no more regular than Lipschitz continuous. Therefore, this standard linearization route is not suitable here.

Instead, our approach relies on the  {\em{monotonicity}} property (\ref{monotonicity_prop}), through a different form of linearization. Specifically, we employ a semi-linearization of the diffusion term in (\ref{weak.formulation.1}):
$$
\left(1-\chi \frac{z}{\lvert \nabla \rho \rvert}\right)_+\nabla \rho.
$$
This operator in monotone in the variable $\nabla \rho$, see (\ref{monotonicity_prop}). Existence of solutions to the {\em{semi-linearized}} problem follows then from the theory of monotone operators, after adding some viscosity terms $\varepsilon( \Delta \rho - \rho)$ to  secure uniform ellipticity. In this framework, fixed-point arguments will be carried out in $L^2$ and require $H^1$-type of estimates for $\rho$. 

The next crucial step involves taking the vanishing viscosity limit, particularly identifying the limit of the nonlinear flux:
$$
\lim_{ \varepsilon \to 0}\left(1-\frac{\chi}{\lvert\nabla \log \rhoe\rvert}\right)_+\nabla \rhoe. 
$$
A common strategy  would be to establish strong convergence (say in $L^2$) of either $\nabla \rhoe$ or $\left(1-\frac{\chi}{\lvert\nabla \log \rhoe\rvert}\right)_+$. Since both terms involve first derivatives of $\rho$, which may be discontinuous, making strong convergence highly non trivial to prove.

 To overcome this, we use the {\em{Boltzmann entropy}} between two viscous  solutions. The Boltzmann relative entropy between two solutions $\rhoe$ and $ \rhod$ is defined as
$$
H[\rhoe, \rhod] = \int_{\Rd} \rhod \left[\frac{\rhoe}{\rhod}\log{\frac{\rhoe}{\rhod}}-\frac{\rhoe}{\rhod}+1\right]\; dx.
$$
This functional bounds the  in $L^1$-distance between $\rhoe$ and $\rhod$. Direct computations show that 
\begin{align*}
\frac{d}{dt} H[\rhoe, \rhod] =&  -\frac{1}{2}\IRd \rhoe\left[\left(1-\frac{\chi \rhoe}{|\na\rhoe|}\right)_+-\left(1-\frac{\chi \rhod}{|\na\rhod|}\right)_+\right]^2\dx\\
	& -\frac{1}{2} \IRd \rhoe\left\vert\na\log\left(\frac{\rhoe }{\rhod}\right)\right\vert^2\left[\left(1-\frac{\chi\rhoe}{\lvert\na\rhoe\rvert}\right)_+\hspace{-2.5mm}+\left(1-\frac{\chi \rhod}{\lvert\na\rhod\rvert}\right)_+\right]\dx \\
 & + \textrm{error terms},
\end{align*}
where the error terms derive from the viscosities $\varepsilon \Delta \rhoe$ and $\delta \Delta \rhod$. As $\varepsilon$ and $\delta$ approach zero, the error terms vanish, and consequently  
$$
-\frac{1}{2}\int_0^T \IRd \rhoe\left[\left(1-\frac{\chi \rhoe}{|\na\rhoe|}\right)_+-\left(1-\frac{\chi \rhod}{|\na\rhod|}\right)_+\right]^2\dx \dt \to 0,
$$
and 
$$
-\frac{1}{2}\int_0^T \IRd \rhoe\left\vert\na\log\left(\frac{\rhoe }{\rhod}\right)\right\vert^2\left[\left(1-\frac{\chi\rhoe}{\lvert\na\rhoe\rvert}\right)_+\hspace{-2.5mm}+\left(1-\frac{\chi \rhod}{\lvert\na\rhod\rvert}\right)_+\right]\dx \dt \to 0. 
$$
Note that $H[\rhoe, \rhod](0) =0$ because $\rhoe$and  $\rhod$ have same initial data. The main point here is that the two limits above will imply that, as $\varepsilon$ and $\delta$ approach zero,  
$$
\left(1-\frac{\chi}{\lvert\nabla \log \rhoe\rvert}\right)_+\nabla \rhoe \rightharpoonup \left(1-\frac{\chi}{\lvert\nabla \log \rho\rvert}\right)_+\nabla \rho,
$$
weakly in $L^2$. The identification of this limit is the content of Section \ref{sec:entropy_limit} and Section \ref{sec4.2}. 
Almost automatically, the relative entropy argument will also provide uniqueness of the solution for the system without viscosity. 


  
  The question of whether one can obtain more regularity is left for future investigation. Based on steady-state solutions of the form (\ref{steadystate}), the maximal regularity one should expect is Lipschitz continuity. The solution that we obtain is bounded and belongs to $L^p(t,T, C^{0,\alpha}(\mathbb{R}^d))$ for any $t>0$. Potential insights might be drawn from existing results on the $p$-Laplacian equation \cite{DB86, GSV10}. However,  the structure of our degenerate set is slightly more complicated than the one of the $p$-Laplacian.  A De Giorgi argument for H\"older regularity might pose challenges, since our diffusion coefficient is degenerate, probably in a set of positive measure, and depends in a nonlinear way on $\nabla \log \rho$. 
  
  Another very interesting open question is the long time dynamics for (\ref{1}). While steady states of the form (\ref{steadystate}) are natural candidates, the set of possible profiles also include all functions satisfying $ | \nabla \log \rho_{\infty} | \le \chi$.

\vspace{0.3cm}

{\bf{Organization of the paper:}} The rest of the paper is organized as follows: in the next section we present the relevant a-priori estimates. The proof of Theorem  \ref{thr.ex} is in Section \ref{Proof section} and Section \ref{sec_viscosity}: in Section \ref{Proof section} we provide the proof of existence of solutions to the model equipped with extra viscosity. In Section \ref{sec_viscosity} we perform the vanishing viscosity limit and show uniqueness of solution.

We will use the decomposition  
$$
f = f_+ - f_-, \quad f_+ := \max\{f,0\},\quad  f_- := -\min\{f,0\}.
$$

\section{A priori estimates}
This section is devoted to a-priori estimates of solutions to  (\ref{weak.formulation.1}).  We combine, without creating confusion, estimates for (\ref{weak.formulation.1}) and for its {\em{viscous}} version, namely  
\begin{align}\label{smooth.eq}
\int_0^T \int_{\Rd} \partial_t \rho_\eps \varphi dxdt = & - \int_0^T \int_{\Rd}\left( 1 - \frac{\chi \rhoe}{|\nabla \rhoe|}\right)_+ \nabla \rhoe \nabla \varphi \;dxdt \\ 
& - \varepsilon \int_0^T \int_{\Rd} \nabla \rhoe \nabla \varphi \;dxdt  - \varepsilon \int_0^T \int_{\Rd}  \rhoe  \varphi \;dxdt. \nonumber
\end{align}
Notice that (\ref{smooth.eq}) coincides with (\ref{weak.formulation.1}) if $\varepsilon =0$.  All our estimates in this section hold for any $\varepsilon \ge 0$.


\begin{lem}[Mass evolution]\label{mass decay lemma} Let $\rhoe$ be a weak solution to (\ref{smooth.eq}) for any $\varepsilon \ge 0$. It holds: 
\begin{align*}
	\int_{\Rd} \rhoe(t)=e^{-\eps t}\int_{\Rd}\rho_{in}.
\end{align*}

\end{lem} 
\begin{proof}
	Testing \eqref{smooth.eq} with the test function $\varphi=\frac{1}{\left(1+\frac{|x|^2}{R^2}\right)^{\alpha/2}}$, with some $\alpha>{d}$, we get 
	\begin{align}\label{mass-proof-1}
		(\int_{\Rd} \rhoe\varphi\dx)(T) = &\;(1+\eps)\int_0^T\hspace{-2mm}\int_{\Rd} \rhoe \Delta\varphi\dx\dt+\chi\int_0^T\hspace{-2mm}\int_{\Rd}\rhoe\frac{\nabla\rhoe}{|\nabla\rhoe|}\nabla\varphi\dx\dt\\
		&\,-\int_0^T\hspace{-2mm}\int_{\Rd}\left(1-\frac{\chi \rhoe}{\lvert\nabla\rhoe\rvert}\right)_-\nabla\rhoe\nabla\varphi\dx\dt-\eps\int_0^T\hspace{-2mm}\int_{\Rd} \rhoe\varphi\dx\dt+\int_{\Rd}\rho_{in}\varphi \dx\nonumber.
	\end{align}
	Here we have rewritten 
	\begin{align}\label{alternative_form}
	\left( 1 - \frac{\chi \rhoe}{|\nabla \rhoe|}\right)_+ \nabla \rhoe  = \nabla \rhoe - \frac{\chi \rhoe }{|\nabla \rhoe|}\nabla \rhoe - \left( 1 - \frac{\chi \rhoe}{|\nabla \rhoe|}\right)_- \nabla \rhoe. 
	\end{align}
Notice that 
\begin{align*}
	-&\int_0^T\hspace{-2mm}\int_{\Rd}\left(1-\frac{\chi \rhoe}{\lvert\nabla\rhoe\rvert}\right)_-\nabla\rhoe\nabla\varphi\dx\dt=\int_0^T\hspace{-2mm}\int_{\Rd\cap\{|\na\rhoe|\leq\chi\rhoe\}}\left(1-\frac{\chi \rhoe}{\lvert\nabla\rhoe\rvert}\right)\nabla\rhoe\nabla\varphi\dx\dt\\
	&\leq\int_0^T\hspace{-2mm}\int_{\Rd\cap\{|\na\rhoe|\leq\chi\rhoe\}}\nabla\rhoe\nabla\varphi\dx\dt-\chi\int_0^T\hspace{-2mm}\int_{\Rd\cap\{|\na\rhoe|\leq\chi\rhoe\}}\rhoe\frac{\nabla\rhoe}{|\na\rhoe|}\nabla\varphi\dx\dt\\
	&\leq {(1+\chi)}\int_0^T\hspace{-2mm}\int_{\Rd\cap\{|\na\rhoe|\leq\chi\rhoe\}}\rhoe|\nabla\varphi|\dx\dt.
\end{align*}
So, introducing the estimate above in \eqref{mass-proof-1} we get
	\begin{align*}
		(\int_{\Rd} \rhoe\varphi\dx)(T)\leq&(1+\eps)\int_0^T\hspace{-2mm}\int_{\Rd} \rhoe \Delta\varphi\dx\dt+(2\chi+1)\int_0^T\hspace{-2mm}\int_{\Rd}\rhoe\left|\nabla\varphi\right|\dx\dt+\int_{\Rd}\rho_{in}\varphi \dx\\
		=&\;(1+\eps)\int_0^T\hspace{-2mm}\int_{\Rd}\rhoe\left[\frac{\alpha(\frac{\alpha}{2}+1)\frac{|x|^2}{R^4}}{\left(1+\frac{|x|^2}{R^2}\right)^{\frac{\alpha}{2}+2}}-\frac{\alpha d\frac{1}{R^2}}{\left(1+\frac{|x|^2}{R^2}\right)^{\frac{\alpha}{2}+1}}\right]\dx\dt\\
		&+3\chi\int_0^T\hspace{-2mm}\int_{\Rd}\frac{\rhoe\alpha\frac{|x|}{R^2}}{\left(1+\frac{|x|^2}{R^2}\right)^{\frac{\alpha}{2}+1}}\dx\dt-\eps\int_0^T\hspace{-2mm}\int_{\Rd} \rhoe\varphi\dx\dt+\int_{\Rd}\rho_{in}\varphi \dx\\
		\leq&\;(1+\eps)\frac{\alpha}{2R}\int_0^T\hspace{-2mm}\int_{\Rd}\frac{\rhoe(\frac{\alpha}{2}+1-d)}{\left(1+\frac{|x|^2}{R^2}\right)^{\frac{\alpha}{2}+1}}\dx\dt+3\chi\frac{\alpha}{R}\int_0^T\hspace{-2mm}\int_{\Rd}\frac{\rhoe}{\left(1+\frac{|x|^2}{R^2}\right)^{\frac{\alpha}{2}}}\\
		&-\eps\int_0^T\hspace{-2mm}\int_{\Rd} \rhoe\varphi\dx\dt+\int_{\Rd}\rho_{in}\varphi \dx.
	\end{align*}
	Now, we choose $\alpha$ such that $\frac{\alpha}{2}+1-d>0$ and get
	\begin{align}
		(\int_{\Rd} \rhoe\varphi\dx)(T)\leq\;\frac{C}{R}\int_0^T\hspace{-2mm}\int_{\Rd} \rhoe\varphi\dx\dt-\eps\int_0^T\hspace{-2mm}\int_{\Rd} \rhoe\varphi\dx\dt+\int_{\Rd}\rho_{in}\varphi \dx,
	\end{align}
	where $C>0$.
	Gronwall's Lemma yields
	\begin{align*}
		(\int_{\Rd} \rhoe\varphi\dx)(T)\leq e^{\frac{C T}{R}}\int_{\Rd} \rho_{in}\varphi\dx\leq e^{\frac{C T}{R}}\int_{\Rd} \rho_{in}\dx.
	\end{align*}
	By monotone convergence we can take the limit $R\rightarrow\infty$ and conclude that $\int\rhoe\leq\int\rho_{in}$.
	
Now that we know that the mass of $\rhoe$ is bounded, we test again \eqref{smooth.eq} with $\varphi$ , but this time we pass directly to the limit $R\rightarrow\infty$ in \eqref{mass-proof-1} to find
\begin{align}\label{mass decay comp}
		(\int_{\Rd} \rhoe \dx)(T)=-\eps\int_0^T\hspace{-2mm}\int_{\Rd}\rhoe\dx+\int_{\Rd}\rho_{in}  \dx.
\end{align}
{The above identity can be rewritten as an initial value problem for an ODE, which can be easily solved, yielding}
\begin{align*}
	(\int_{\Rd} \rhoe \dx)(T)=e^{-\eps T}\int_{\Rd}\rho_{in}\dx, 
\end{align*}
which finishes the proof. 
\end{proof}
\begin{prop}[$L^p$, $H^1$ estimates]\label{Proposition LP} Let $\rhoe$ be a weak solution to (\ref{smooth.eq}) for any $\varepsilon \ge 0$, then

\begin{align*}
 \|\rhoe\|^p_{L^{\infty}((0,T);L^p(\Rd))}+ 
 \|\nabla\rhoe^{\frac{p}{2}}\|^2_{L^{2}((0,T)\times\Rd)}\leq C,
\end{align*} 
where $C$ only depends on $\|\rho_{in}\|_{L^p}, p$, $T$, and $\chi$.
\end{prop}
\begin{proof}
	First, notice that testing \eqref{smooth.eq} with $\rhoe^{p-1}$we get that $\int_{\Rd}\rhoe^p(t)\leq\int_{\Rd}\rho_{in}^p$. Then, we compute
	\begin{align*}
		\frac{1}{p}\int_{\Rd} \rhoe^p(t)\leq& -(p-1)\int_0^T\hspace{-2mm}\int_{\Rd}\rhoe^{p-2}\lvert\nabla\rhoe\rvert^2+\chi(p-1)\int_0^T\hspace{-2mm}\int_{\Rd}\rhoe^{p-1}\lvert\nabla\rhoe\rvert+\frac{1}{p}\int_{\Rd} \rho_{in}^p\\
		= &-\frac{2(p-1)}{p^2}\int_0^T\hspace{-2mm}\int_{\Rd}\lvert\nabla\rhoe^{\frac{p}{2}}\rvert^2+\chi^2(p-1)\int_0^T\hspace{-2mm}\int_{\Rd}\rhoe^p+\frac{1}{p}\int_{\Rd} \rho_{in}^p,
	\end{align*}
which yields
	\begin{align*}
		\|\rhoe\|^p_{L^{\infty}(0,T;L^p(\Rd))}+\|\nabla\rhoe^{\frac{p}{2}}\|^2_{L^{2}(0,T;L^2(\Rd))}\leq C,
	\end{align*}
	with the right hand side independent of $\eps$.
\end{proof}
\begin{prop}[Moment Evolution]\label{moment lemma}  Let $\rhoe$ be a weak solution to (\ref{smooth.eq}) for any $\varepsilon \ge 0$, then
\begin{align*}
	\int_{\Rd}\rhoe(t)(1+|x|^2) \dx\leq e^{C\;t}\int_{\Rd}\rho_{in}(1+|x|^2) \dx,
\end{align*}
where $C$ is independent of $\eps$.
\begin{proof}
We test the equation with $(1+|x|^2)\varphi$, where $\varphi=\frac{1}{\left(1+\frac{|x|^2}{R^2}\right)^{\alpha/2}}$ as in Lemma \ref{mass decay lemma}. Then, by Young's inequality, noticing that $|\na\varphi|, |\Delta\varphi|\leq\varphi$ and $0\leq\left(1-\frac{\chi\rhoe}{|\na\rhoe|}\right)_-|\na\rhoe|\leq2\chi\rhoe$, we get
\begin{align*}
	(\int_{\Rd}\rhoe(1+|x|^2)\varphi \dx)(T)\leq C\int_0^T\hspace{-2mm}\int_{\Rd}\rhoe(1+|x|^2)\varphi \dx\dt+\int_{\Rd}\rho_{in}(1+|x|^2)\varphi \dx.
\end{align*}
Gronwall's lemma gives 
\begin{align*}
	(\int_{\Rd}\rhoe(1+|x|^2)\varphi \dx)(T)\leq e^{CT}\int_{\Rd}\rho_{in}(1+|x|^2)\varphi \dx.
\end{align*}
By dominated convergence we pass to the limit in $R$ and get the result.
\end{proof}
\end{prop}
\begin{prop}[ Estimate for $\|\nabla\sqrt{\rhoe}\|_{L^2((0,T)\times\Rd)}$]\label{grad.log.estimate} Let $\rhoe$ be a weak solution to (\ref{smooth.eq}) for any $\varepsilon \ge 0$, then
\begin{align*}
	(\int_{\Rd }\rhoe\log\rhoe\dx)(T)+\int_0^T\int_{\Rd}\frac{\lvert\nabla\rhoe\rvert^2}{\rhoe}\dx\dt\leq \bar{C}[\rho_{in}],%
\end{align*}
where $\bar{C}$ depends on the moment, entropy and $L^2$ norm of the initial data.
\end{prop} 
\begin{proof}
	Let us test \eqref{smooth.eq} with the test function $\varphi=\log(\rhoe+\sigma)-\log{\sigma}$, with $0<\sigma<<1$. After some computations we get
	\begin{align*}
		&\left(\IRd\left[(\rhoe+\sigma)\log(\rhoe+\sigma)-\rhoe\log\sigma\right]\dx\right)(T)+\eps\int_0^T\hspace{-2mm}\int_{\Rd}\rhoe\left[\log(\rhoe+\sigma)-\log{\sigma}\right]\dx\dt\\
		&\qquad\leq\IRd\left[(\rho_{in}+\sigma)\log(\rho_{in}+\sigma)-\rho_{in}\log\sigma\right]\dx,
	\end{align*} 
	which implies
	\begin{align*}
	\left(\IRd(\rhoe+\sigma)\log(\rhoe+\sigma)\dx\right)(T)\leq&\IRd  (\rho_{in}+\sigma)\log(\rho_{in}+\sigma)\dx \\
    &-\eps\int_0^T\hspace{-2mm}\int_{\Rd}\rhoe\log{(\rhoe+\sigma)}\dx\dt\\
	&+\log\sigma\left(\IRd\rhoe\dx-\IRd\rho_{in}\dx+\eps\int_0^T\hspace{-2mm}\int_{\Rd}\rhoe\dx\dt\right).
	\end{align*} 
	The last term in the previous inequality is zero by \eqref{mass decay comp}. Taking $\sigma\rightarrow0$, by monotone and dominated convergence, we get
	\begin{align*}
		\left(\IRd\rhoe\log\rhoe\dx\right)(T)\leq&\IRd  \rho_{in}\log\rho_{in}\dx-\eps\int_0^T\hspace{-2mm}\int_{\Rd}\rhoe\log{\rhoe}\dx\dt.
	\end{align*} 
	In particular, 
	\begin{align*}
		\left(\IRd\rhoe|\log\rhoe|\dx\right)(T)\leq e^{\eps T}C[\rho_{in}],
	\end{align*}
	where $C[\rho_{in}]$ is a constant that depends on $\IRd\rho_{in}|\log{\rho_{in}}|\dx$, {the mass} and the second moment of $\rho_{in}$.
	
	To bound $\int_0^T\IRd\frac{|\na\rhoe|^2}{\rhoe}$, we define $h[\rhoe](t):=\left(\IRd\rhoe\log(\rhoe)\dx\right)(t)$. Using the fact that 
	\begin{align*}
		\left(1-\frac{\chi\rhoe}{|\na\rhoe|}\right)_+ =\left(1-\frac{\chi\rhoe}{|\na\rhoe|}\right)+\left(1-\frac{\chi\rhoe}{|\na\rhoe|}\right)_-,
	\end{align*}
	and repeating the same computations as before, we obtain
	\begin{align*}
		h[\rhoe](T)+\frac{1}{2}\int_0^T\hspace{-2mm}\int_{\Rd}\frac{|\na\rhoe|^2}{\rhoe+\sigma}\dx\dt\leq\frac{\chi^2}{2}\int_0^T\hspace{-2mm}\int_{\Rd}\rhoe\dx\dt-\eps\int_0^T h[\rhoe](t)\dt+h[\rho_{in}],
	\end{align*}
	which implies
	\begin{align*}
		\frac{1}{2}\int_0^T\hspace{-2mm}\int_{\Rd}\frac{|\na\rhoe|^2}{\rhoe+\sigma}\dx\dt\leq\frac{\chi^2}{2}\int_0^T\hspace{-2mm}\int_{\Rd}\rhoe\dx\dt+ e^{\eps T}C[\rho_{in}].
	\end{align*}
	\normalcolor 
	\color{red}
\normalcolor
\end{proof}

\begin{lem}[$\|\nabla\rhoe\|_{L^p}$ estimate]\label{Proposition Lp gradient}  Let $\rhoe$ be a weak solution to (\ref{smooth.eq}) for any $\varepsilon \ge 0$, then
	\begin{align*}
		\|\nabla\rhoe\|_{L^p((t,T)\times \Rd)}< \frac{1}{t^{1/2}}C(\|\rho_{in}\|_{L^p}),\quad \forall  \; t>0.
	\end{align*}
\end{lem}
\begin{proof}
	Define $f$ as 
	\begin{align*}
	f := \left(1-\frac{\chi\rhoe}{\lvert\nabla\rhoe\rvert}\right)_+ \nabla \rhoe - \nabla \rhoe.
	\end{align*} 
	Note that each component of $f$ belongs to $L^\infty(0,T,L^q(\mathbb{R}^d))$ for any 
 $q\in [1,+\infty)$. In fact, in the domain where $|\nabla\rhoe| \le \chi \rhoe$, the vector function $f$ coincides with $\nabla \rhoe$, which is bounded by $\chi \rhoe$.  In the complementary domain, we have that 
	$$
	\left(1-\frac{\chi\rhoe}{\lvert\nabla\rhoe\rvert}\right)_+  = \left(1-\frac{\chi\rhoe}{\lvert\nabla\rhoe\rvert}\right),
	$$
	and $|f | \le  \chi  \rhoe$.  With this in mind, consider the function $u_i$, solution to 
	$$
	{u_i}_t - (1+\varepsilon) \Delta  u_i + \varepsilon u_i = f_i, \quad u_i(x,0) =0, \quad i = 1, ...,d.
	$$
	Parabolic regularity theory tells us that $u_i \in L^q(0,T,W^{2,q})$, with bounds uniformly in $\varepsilon$. Direct computations shows that $\rhoe$ can be written as the sum of $\textrm{div}(u) + g$, where $g$ solves 
	$$
	g_t = (1+\varepsilon) \Delta g - \varepsilon g, \quad g(x,0) = \rho_{in}.
	$$
	Since $\rho_0 \in L^p(\mathbb{R}^d)$,  we have that 
	$$
	\sup_{t>0} \|\nabla g\|_{L^p} \le\frac{c}{ t^{1/2}} \|\rho_0\|_{L^p},
	$$
	with $c$ independent on $\varepsilon$.  Summarizing, $\textrm{div}(u) \in L^p(0,T, W^{1,p})$, $g \in L^\infty (t,T, W^{1,p})$,  and the thesis follows.

\end{proof}

\begin{lem}\label{Lemma.DeGiorgi}  Let $\rhoe$ be a weak solution to (\ref{smooth.eq}) for any $\varepsilon \ge 0$, then {for every $t>0$ it holds}
\begin{equation*}
\|\rhoe\|_{L^{\infty}({t},T;\Rd)}\leq C {T^{\frac{1}{p}}\left(1 + t^{-\frac{d+2}{2p}}\right)}
\|\rho_{in}\|_{L^p(\Rd)},
\end{equation*}
where $C$ only depends on $p$, $\chi$ and other universal constants.
\end{lem}
\begin{proof}
	Let us consider the following sequence of times 
	\begin{align*}
		T_k=\tau(1-2^{-k}),
	\end{align*}
	and the following sequence of energies
	\begin{align*}
		U_k=\int_{T_k}^{T}\int_{\Rd} \rho_{\eps, k}^p\dx\dt, \quad\mbox{where}\quad \rho_{\eps, k}=(\rhoe-C_k)_+,\quad\mbox{and}\quad C_k=M(1-2^{-k}),
	\end{align*}
	with $M>0$ to fix later. Then, we test \eqref{smooth.eq} with $\rho_{\eps, k+1}^{p-1}$ and we integrate in time in $(t_1, t_2)$, where $0\leq T_k\leq t_1\leq T_{k+1}\leq t_2\leq T$. We find that, using (\ref{alternative_form}), we get 
	\begin{align*}
	\frac{1}{p}\int_{t_1}^{t_2}\int_{\Rd} &\partial_t(\rho_{\eps, k+1}^p)\dx\dt\\ 
    &\leq -\int_{t_1}^{t_2}\int_{\Rd}\nabla \rhoe\nabla(\rho_{\eps, k+1}^{p-1})\dx\dt+\int_{t_1}^{t_2}\int_{\Rd}\frac{\chi \rho}{|\nabla \rhoe|}\nabla(\rho_{\eps, k+1}^{p-1})\na \rhoe\dx\dt.
	\end{align*}
	After some computations, from the expression above we get
	\begin{align}\label{DeGiorgi.1}
		\frac{1}{p}\int_{t_1}^{t_2}\hspace{-2mm}\int_{\Rd}\partial_t(\rho_{\eps, k+1}^p)\dx\dt+\frac{(p-1)}{p^2}\int_{t_1}^{t_2}\hspace{-2mm}\int_{\Rd}|\nabla \rho_{\eps, k+1}^{\frac{p}{2}}|^2\dx\dt\leq\frac{(p-1)}{2}\int_{t_1}^{t_2}\hspace{-2mm}\int_{\Rd}\rho_{\eps, k+1}^p\dx\dt\nonumber\\
		+2\chi^2C_{k+1}^2(p-1)\int_{t_1}^{t_2}\hspace{-2mm}\int_{\Rd} \rho_{\eps, k+1}^{p-2}\dx\dt.
	\end{align}
	Now, we notice that, in the set $\rho_{\eps, k+1}>0$, $\rho_{\eps, k}=\rho_{\eps, k+1}+C_{k+1}-C_{k}$, i.e. that $\frac{\rho_{\eps, k}}{C_{k+1}-C_k}=\frac{\rho_{\eps, k+1}}{C_{k+1}-C_k}+1>1$. Then, for any $\alpha>0$ we have $\left(\frac{\rho_{\eps, k}}{C_{k+1}-C_k}\right)^{\alpha}>1$ and, consequently $\left(\frac{1}{C_{k+1}-C_k}\right)^{\alpha}\rho_{\eps, k}^{1+\alpha}>\rho_{\eps, k+1}$. Now, let us choose $\alpha=\frac{2}{p-2}$ to get 
	\begin{align*}
		\rho_{\eps, k+1}\leq\left(\frac{1}{C_{k+1}-C_k}\right)^{\frac{2}{p-2}}\rho_{\eps, k}^{\frac{p}{p-2}}\mathds{1}_{\{\rho_{\eps, k+1}>0\}}.
	\end{align*}
	Introducing this expression in \eqref{DeGiorgi.1} we get
	\begin{align*}
		&\frac{1}{p}\left(\int_{\Rd}\rho_{\eps, k+1}^p\dx\right)(t_2)+\frac{(p-1)}{p^2}\int_{t_1}^{t_2}\hspace{-2mm}\int_{\Rd}|\nabla \rho_{\eps, k+1}^{\frac{p}{2}}|^2\dx\dt\leq \frac{(p-1)}{2}\int_{t_1}^{t_2}\hspace{-2mm}\int_{\Rd}\rho_{\eps, k+1}^p\dx\dt\\
		&+2\chi^2(p-1)\left(\frac{C_{k+1}}{C_{k+1}-C_k}\right)^2\int_{t_1}^{t_2}\hspace{-2mm}\int_{\Rd} \rho_{\eps,k}^{p}\dx\dt+\frac{1}{p}\left(\int_{\Rd}\rho_{\eps, k+1}^p\dx\right)(t_1).
	\end{align*}
	Then, averaging $t_1$ in $(T_k,T_{k+1})$ and taking the supremum in $t_2\in(T_{k+1},T)$ we have
	\begin{align*}
			E_{k+1}& :=\sup_{t\in[T_{k+1}, T]}\left(\int_{\Rd}\rho_{\eps, k+1}^p\dx\right)+\int_{T_{k+1}}^{T}\int_{\Rd}|\nabla \rho_{\eps, k+1}^{\frac{p}{2}}|^2\dx\dt\\
			&\;\;\leq C\left(1+2^k+\frac{2^k}{{\tau}}\right)\int_{T_{k}}^{T}\int_{\Rd}\rho_{\eps,k}^p\dx\dt,
	\end{align*}
where $C$ only depends on $p$ and $\chi$. From now on, this $C$ will change from line to line for convenience, always depending only on  $p$ and $\chi$.
	Now, notice that $E_{k+1}$ controls $\rho_{\eps,k+1}$ in $L^{\infty}(T_{k+1}, T;L^p(\Rd))$ and, thanks to Sobolev,
	$L^{p}(T_{k+1}, T;L^{\frac{dp}{d-2}}(\Rd))$, so, by interpolation
	\begin{align}\label{DeGiorgi-computations-1}
		\|\rho_{\eps,k+1}\|^p_{L^{\frac{p(d+2)}{d}}(\R^d\times (T_{k+1},T))}\leq CE_{k+1}\leq C'\left(1+2^k+\frac{2^k}{{\tau}}\right)\int_{T_{k}}^{T}\int_{\Rd}\rho_{\eps,k}^p\dx\dt.
	\end{align} 
	Then, Chebyshev's inequality yields:
	\begin{align*}
		U_{k+1}:&=\int_{T_{k+1}}^{T}\int_{\Rd} \rho_{\eps, k+1}^p\dx\dt\\
		&\leq\left(\int_{T_{k+1}}^{T}\int_{\Rd}\rho_{\eps, k+1}^{\frac{p(d+2)}{d}}\dx\dt\right)^{\frac{d}{d+2}}\left(\int_{T_{k+1}}^{T}\int_{\Rd}\mathds{1}_{\{\rho_{\eps, k+1}>0\}}\dx\dt\right)^{\frac{2}{d+2}}\\
		&=\left(\int_{T_{k+1}}^{T}\int_{\Rd}\rho_{\eps, k+1}^{\frac{p(d+2)}{d}}\dx\dt\right)^{\frac{d}{d+2}}\left(\int_{T_{k+1}}^{T}\int_{\Rd}\mathds{1}_{\{\rho_{\eps,k}>\frac{M}{2^{k+1}}\}}\dx\dt\right)^{\frac{2}{d+2}}\\
		&\leq CE_{k+1}\left(\frac{2^{k+1}}{M}\right)^{\frac{2p}{d+2}}\left(\int_{T_{k}}^T\int_{\R^d}\rho_{\eps, k}^p\dx\dt\right)^{\frac{2}{d+2}}\\
		&\leq \left(1{+\frac{1}{\tau}}\right)\frac{C^k}{M^{\frac{2p}{d+2}}}U_k^{\beta},
	\end{align*}
	where $\beta=1+\frac{2}{d+2}$, using \eqref{DeGiorgi-computations-1} to obtain the second and the third inequality. {From \cite[Lemma 7.1]{Giusti2003} it follows}  
 that $U_k\rightarrow0$ as $k\rightarrow \infty$, i.e. that $\rhoe\leq M$, if
	\begin{align*}
		\left(1+\frac{1}{{\tau}}\right)^{\frac{d+2}{2}}\frac{C}{M^p}U_0<1.
	\end{align*}
	So, we fix $M$ in order to 
	\begin{align*}
		C\left(1+\frac{1}{{\tau}}\right)^{\frac{d+2}{2p}} U_0^{1/p}\sim M.
	\end{align*}
Finally, we conclude {from Proposition \ref{Proposition LP}}
\begin{align*}
	\|\rhoe\|_{L^{\infty}(0,T;\Rd)}\leq M\sim C\left(1+\frac{1}{{\tau}}\right)^{\frac{d+2}{2p}} \|\rhoe\|_{L^p(0,T;\Rd)}\leq C {T^{\frac{1}{p}}}\left(1+\frac{1}{{\tau}^{\frac{d+2}{2p}}}\right) \|\rho_{in}\|_{L^p(\Rd)}.
\end{align*}
\end{proof}

We conclude this section by recalling a compactness criteria for functions defined in unbounded domains:
\begin{lem}[Compactness Lemma]\label{compactness.lemma}
	 Let $\{f_n\}$ be a bounded sequence that belongs to $L^2(0,T;H^1(\Rd))\cap 
  L^2(0,T;\Rd,\sqrt{1+|x|^2}\dx)$ and such that 
  $\pa_t f_{n}$ is bounded in $L^2(0,T;H^{-1}(\Rd))$. We have
	\begin{align*}
		f_n\rightarrow f\qquad\mbox{strongly in }\;L^2(0,T;\Rd), \qquad\mbox{as }n\rightarrow\infty.
	\end{align*}

\end{lem}
\begin{proof}
	As $\{f_n\}$ is uniformly bounded in $L^2(0,T;H^1(\Rd))$ we have that
\begin{align*}
	f_n\rightharpoonup f\quad\mbox{weakly in }L^2(0,T;L^{\frac{2d}{d-2}}(\Rd)).
\end{align*}
Then, for any $R>0$, as $\pa_tf_{n}\in L^2(0,T;H^{-1}({B_R}))$ and $H^1(B_R)\subset\subset L^{q}(B_R)$ with $2\leq q<{\frac{2d}{d-2}}$, by Aubin-Lions Theorem (Theorem 5 in \cite{Simon86}), we have that a subsequence, $\{f_n^R\}$, converges strongly to $f$ as follows
\begin{align*}
	f_n^R\rightarrow f\quad\mbox{strongly in }L^2(0,T;L^{q}(B_R)).
\end{align*}
By a Cantor diagonal argument, we have a new subsequence $f_n$, independent of $R$, such that
\begin{align*}
	f_n\rightarrow f\quad\mbox{strongly in }L^2(0,T;L^{q}(B_R)),
\end{align*}
and consequently,
\begin{align*}
	&f_n\rightarrow f\quad\mbox{a.e. in }[0,T]\times \Rd.
\end{align*}
By Fatou's Lemma, we also have
\begin{align*}
	\int_0^T\IRd f^2\sqrt{1+|x|^2}\dx\dt\leq \liminf\limits_{\eps\rightarrow0}\int_0^T\IRd f_n^2\sqrt{1+|x|^2}\dx\dt<\infty.
\end{align*}
Now, setting $q=2$, for any $\delta>0$, there exist sufficiently large $n$ and $R$ such that
\begin{align*}
	&\int_0^T\IRd\lvert f_n-f\rvert^2\dx\dt=\int_0^T\int_{B_R}\lvert f_n-f\rvert^2\dx\dt+\int_0^T\int_{\Rd\setminus B_R}\lvert f_n-f\rvert^2\dx\dt\\
	&\leq\frac{\delta}{2}+\frac{1}{\sqrt{1+R^2}}\int_0^T\int_{\Rd\setminus B_R}\lvert f_n-f\rvert^2 \sqrt{1+|x|^2}\dx\dt\leq\delta,
\end{align*}
and the thesis follows.
\end{proof}

\section{Existence for the viscous problem} \label{Proof section}
This section is devoted to the proof of existence of our equation equipped with extra regularizing terms $-\varepsilon \Delta \rho + \varepsilon \rho$. More precisely, we will show existence of weak solution to 
\begin{align}\label{smooth.eq_bis}
\int_0^T \int_{\Rd} \partial_t \rho_\eps \varphi dxdt = & - \int_0^T \int_{\Rd}\left( 1 - \frac{\chi \rhoe}{|\nabla \rhoe|}\right)_+ \nabla \rhoe \nabla \varphi \;dxdt \\ 
& - \varepsilon \int_0^T \int_{\Rd} \nabla \rhoe \nabla \varphi \;dxdt  - \varepsilon \int_0^T \int_{\Rd}  \rhoe  \varphi \;dxdt, \nonumber
\end{align}
 with initial data $\rho_\eps(\cdot , 0) = \rho_{in,\eps}(\cdot)$, where $\rho_{in,\eps}(\cdot)$ is an approximation of $\rho_{in}$ as smooth as one wishes. 
 
 The construction of solutions starts with a {\em{semi-linearization}} of \eqref{smooth.eq_bis}. We replace $\left( 1 - \chi  \frac{\rhoe}{|\nabla \rhoe|}\right)_+$ with $\left( 1 - \chi \frac{z_+}{|\nabla \rhoe|}\right)_+$ so that the diffusion term 
 $$
 \left( 1 - \chi \frac{z_+}{|\nabla \rhoe|}\right)_+ \nabla \rhoe
 $$
 is monotone with respect to $\nabla \rhoe$. Solution of this semilinearized  problem follows from the theory of monotone operators (Section \ref{sec3.1}). 
Then, the Leray-Schauder fixed point theorem will yield existence of a solution to \eqref{smooth.eq_bis} (Section \ref{sec3.2}).

\subsection{The semi-linearized equation}\label{sec3.1}
For given $\eps$, $\tilde \sigma\in[0, 1]$ and $z\in L^2(0, T; L^2(\Rd))$, let us consider
	\begin{align}\label{1.ugly.reg}
		\int_0^T\int_{\Rd}\partial_t\rhoe\,\varphi\dx\dt=&-\tilde \sigma\int_0^T\int_{\Rd}\left(1-\frac{\chi z_+}{\lvert\nabla\rhoe\rvert}\right)_+\nabla\rhoe\nabla\varphi\dx\dt\nonumber\\
		&-\eps\int_0^T\int_{\Rd}\nabla\rhoe\nabla\varphi\dx\dt-\eps\int_0^T\int_{\Rd}\rhoe\varphi\dx\dt,%
	\end{align}
	\normalsize
	$\forall\varphi\in  L^2(0, T; H^1(\Rd))$, where $\rhoe(0, \cdot)=\rho_{{\rm in},\eps}$ and $\rho_{{\rm in},\eps}$ is a $W^{1,p}\cap L^{\infty}(\Rd)$ approximation of $\rho_{{\rm in}}$ such that $\rho_{{\rm in},\eps}\rightarrow \rho_{{\rm in}}$ strongly in $L^p(\Rd)$ when $\eps\rightarrow0$. 
	
	We show the existence of solutions to (\ref{1.ugly.reg}) via the theory of monotone operators, concretely, \cite[Theorem 30.A]{ZeidlerNonlinearIIB}. We define the following operator $A$ as 
	\begin{align*}
		\langle A[u](t), \varphi\rangle \; := \; &\tilde \sigma\int_{\Rd}\left(1-\frac{\chi z_+}{\lvert\nabla u(t, x)\rvert}\right)_+\nabla u(t, x)\nabla\varphi\dx\\
		&+\eps\int_{\Rd}\nabla u(t,x)\nabla\varphi\dx+\eps\int_{\Rd} u(t, x)\varphi \dx,
	\end{align*}
	which must satisfy the following properties:
	
	\emph{1. Monotonicity. }. The operator $A$ is monotone, since $\forall \; u,v\in H^1(\Rd)$ we have
	\begin{align*}
		\langle A[u]-A[v],u-v \rangle=&\tilde \sigma\int_{\Rd}\left[\left(1-\frac{\chi z_+}{\lvert\nabla u\rvert}\right)_+\nabla u-\left(1-\frac{\chi z_+}{\lvert\nabla v\rvert}\right)_+\nabla v\right]\nabla(u-v)\dx\\
		&+\eps\|u-v\|^2_{H^1(\Rd)}\geq0.
	\end{align*}
	
	\emph{2. Hemicontinuity. } Let us check that the map
	\begin{align*}
		t\mapsto\langle A[u+t v],\varphi \rangle
	\end{align*}
	is continuous in $t\in[0,1]$ for all $u, v, \varphi \in H^1(\Rd)$. Notice that $$\left\lvert\left(1-\frac{\chi z_+}{\lvert\nabla(u+t v)\rvert}\right)_+\nabla(u+t v)\nabla\varphi\right\rvert\leq\lvert\nabla u\rvert\lvert\nabla\varphi\rvert+\lvert\nabla v\rvert\lvert\nabla\varphi\rvert$$ so
	\small
	\begin{align*}
		\lim_{t\rightarrow t_0}\langle A[u+t v], \varphi\rangle=&\tilde \sigma\int_{\Rd}\left(1-\frac{\chi z_+}{\lvert\nabla(u+t_0 v)\rvert}\right)_+\nabla(u+t_0v)\nabla\varphi\dx+\eps\int_{\Rd}\nabla(u+t_0v)\nabla\varphi\dx\\
		&+\eps\int_{\Rd}(u+t_0v)\varphi\dx=\langle A(u+t_0 v), \varphi\rangle,
	\end{align*}
	\normalsize
	by dominated convergence. 
	
	\emph{3. Coerciveness. } We just compute
	\begin{align*}
		\langle A[u], u\rangle=\tilde \sigma\int_{\Rd}\left(1-\frac{\chi z_+}{\lvert\nabla u\rvert}\right)_+\lvert\nabla u\rvert^2\dx+\eps\|u\|^2_{H^1(\Rd)}\geq \eps\|u\|_{H(\Rd)}^2,
	\end{align*}
	so $A$ is coercive.
	
	\emph{4. Growth condition. }For this we have to prove that $\|A[u]\|_{H^*(\Rd)}\leq C\|u\|_{H(\Rd)}$, but, by Hölder,
	\begin{align*}
		\|A[u]\|_{H^*(\Rd)}=&\sup_{\substack{\varphi\in H^1(\Rd)\\ \|\varphi\|\leq 1}}\langle A[u], \varphi\rangle\\
		=&\sup_{\substack{\varphi\in H^1(\Rd)\\ \|\varphi\|\leq 1}}\int_{\Rd}\left[\tilde \sigma\left(1-\frac{\chi z_+}{\lvert\nabla u\rvert}\right)_+\nabla u\nabla\varphi+\eps\nabla u\nabla\varphi+\eps u\varphi\right]\dx\\
		\leq& \;C \|u\|_{H(\Rd)}.
	\end{align*}
	
	\emph{5. Measurability. } By assumption.

	Then, by \cite[Theorem 30.A]{ZeidlerNonlinearIIB}, there exists a unique weak solution $\rhoe\in L^2(0,T;H^1(\Rd))$ to \eqref{1.ugly.reg} {with $\pa_t\rho_\eps \in L^2(0,T; H^{-1}(\R^d))$ satisfying the initial condition in the sense of the strong $L^2(\R^d)$ convergence:
 \begin{equation}\label{in.c.ugly}
 \rhoe(t)\to \rho_{in,\eps}\quad\mbox{strongly in $L^2(\R^d)$ as $t\to 0$.}
 \end{equation}
 Therefore} we can define the mapping 
\begin{align*}
\mathcal{T}:(z, \tilde \sigma)\in L^2((0,T)\times\Rd)\times[0,1]\mapsto \rhoe\in L^2((0,T)\times\Rd),
\end{align*}
where $\rhoe$ is the unique solution to \eqref{1.ugly.reg}. 
	
\subsection{The nonlinear problem} \label{sec3.2}

In this subsection we show that our map $\mathcal{T}$ satisfies the assumptions of the Leray-Schauder fixed point theorem. Namely we have to show that $\mathcal{T}$ is (i) continuous, (ii) compact, and (iii) constant when $\tilde \sigma=0$.

 First, using $\rhoe$ as test function in \eqref{1.ugly.reg} we get
	\begin{align*}
		\frac{1}{2}\int_0^T\int_{\Rd}\partial_t(\rhoe^2)\dx\dt=&-\tilde \sigma\int_0^T\int_{\Rd}\left(1-\frac{\chi (z_++\eps)}{\lvert\nabla\rho(t, x)\rvert}\right)_+|\nabla\rhoe|^2\dx\dt\\
		&-\eps\int_0^T\int_{\Rd}|\nabla\rhoe|^2\dx\dt-\eps\int_0^T\int_{\Rd}\rhoe^2\dx\dt,
	\end{align*}
	which implies: 
	\begin{align*}
		\frac{1}{2}\|\rhoe\|^2_{L^{\infty}(0,T; L^2(\Rd))}+\eps\|\rhoe\|^2_{L^2(0,T; H^1(\Rd))}\leq \int_{\Rd}\rho_{{\rm in},\eps}^2 \; dx.
	\end{align*}
	Thus, the $L^2(0,T;H^1(\Rd))$ norm of any solution $\rhoe$ has a uniform bound. Now, we test \eqref{1.ugly.reg} with $\varphi=\rhoe\sqrt{1+|x|^2}\phi_R$, where $\phi_{R}\in C^{\infty}_c(B_{2R})$ is such that $\phi_R\equiv1$ in $B_R$ and $|\nabla\phi_R|\leq \frac{C}{R}$ for some $C>0$. Doing this we get
	\begin{align*}
		\frac{1}{2}\int_{\Rd}\rhoe^2\sqrt{1+\lvert x\rvert^2}\phi_R\dx\leq&\;\;	\frac{1}{2}\int_{\Rd}\rho_{{\rm in},\eps}^2\sqrt{1+\lvert x\rvert^2}\dx+(\tilde \sigma+\eps)\int_0^T\int_{B_{2R}}\frac{\rhoe|\nabla\rhoe|\cdot |x|}{\sqrt{1+\lvert x\rvert^2}}\dx\dt\\
		&+(\tilde \sigma+\eps)\int_0^T\int_{B_{2R}}|\na\rhoe|\rhoe\sqrt{1+|x|^2}|\na\phi_R|\dx\dt.
	\end{align*}
	Then, using $\frac{\lvert x\rvert}{\sqrt{1+\lvert x\rvert^2}}\leq1$, Young's inequality and that $|\nabla\phi_R|\leq \frac{C}{R}$, we get
	\begin{align*}
		\int_0^T\int_{\Rd}\rhoe^2\sqrt{1+\lvert x\rvert^2}\phi_R\dx\dt\leq&	\int_0^T\int_{\Rd}\rho_{{\rm in},\eps}^2\sqrt{1+\lvert x\rvert^2}\dx\dt+c\|\rhoe\|^2_{L^2(0,T;H^1(\Rd))}\\
		& + \frac{C\sqrt{1+4|R|^2}}{R}\int_0^T\int_{\Rd} \rhoe|\na\rhoe|\dx\dt\\
		\leq& \; C(\rho_{{\rm in},\eps}).
	\end{align*}
	Taking the limit $R\rightarrow\infty$ we get that solutions to \eqref{1.ugly.reg} are bounded in $ L^2(0,T;H^1(\Rd))\cap L^2(0,T;\Rd, \sqrt{1+|x|^2}\dx)$. Also, for any $\varphi\in L^2(0,T,H^1(\Rd))$ we have
	\begin{align}\label{time_derivative_bound}
		\left\lvert\int_0^T\int_{\Rd} \partial_t\rhoe(t)\varphi\dx\dt\right\rvert\leq C\|\rhoe(t)\|_{H^1(\Rd)}\|\varphi\|_{L^2(0,T,H^1(\Rd))}\leq C\|\varphi\|_{L^2(0,T,H^1(\Rd))}.
	\end{align}
	By Lemma \ref{compactness.lemma}, $\mathcal{T}$ maps bounded sets to precompact sets.
	
	Now, we have to prove that the map is continuous, i.e. that $\lim\limits_{n\rightarrow\infty}\mathcal{T}(\tilde \sigma_n,z_n)=\mathcal{T}(\tilde \sigma, z)$ when $z_n\rightarrow z$ strongly in $L^2(0,T,L^2(\Rd))$ and $\tilde \sigma_n\rightarrow \tilde \sigma$ in $[0,1]$. Is easy to see that $\lim\limits_{n\rightarrow\infty}\mathcal{T}(\tilde \sigma_n, z_n)=\lim\limits_{n\rightarrow\infty}\mathcal{T}(\tilde \sigma, z_n)$. Then, as the map is compact, it only remains to prove
	\begin{align*}
		\lim_{n\rightarrow\infty} \int_0^T\int_{\Rd}\left(1-\frac{\chi (z_n)_+}{\lvert\nabla\rho_{\eps,n}\rvert}\right)_+\nabla\rho_{\eps,n}\nabla\varphi\dx\dt=\int_0^T\int_{\Rd}\left(1-\frac{\chi z_+}{\lvert\nabla\rhoe\rvert}\right)_+\nabla\rhoe\nabla\varphi\dx\dt,
	\end{align*}
	but as $\left(1-\frac{\chi (z_n)_+}{\lvert\nabla\rho_{\eps,n}\rvert}\right)_+\nabla\rho_{\eps,n}$ is uniformly bounded in $L^2(0,T,L^2(\Rd))$, it has a weakly convergent subsequence and we only need $\nabla\rho_{\eps,n}\rightarrow\nabla\rhoe$ a.e. $x\in\Rd$. To show this convergence, we prove that $\{\na\rho_{\eps,k}\}$ is a Cauchy sequence in $L^2(0,T;L^2)$:	let $\rho_{\eps,n}$ and $\rho_{\eps,k}$ solutions to \eqref{1.ugly.reg} with $z_n$ and $z_k$ respectively. We subtract the equation for $\rho_{\eps,k}$ from the one of $\rho_{\eps,n}$ and get, after some manipulations:
	{ \begin{align*}
			\int_0^T\hspace{-2mm}\int_{\Rd}\partial_t &(\rho_{\eps,n}-\rho_{\eps,k})\varphi\dx\dt\\
    =& -\tilde \sigma\int_0^T\hspace{-2mm}\int_{\Rd}\left[\left(1-\frac{\chi (z_n)_+}{\lvert\nabla\rho_{\eps,n}\rvert}\right)_+\na\rho_{\eps,n}-\left(1-\frac{\chi (z_k)_+}{\lvert\nabla\rho_{\eps,k}\rvert}\right)_+\na\rho_{\eps,k}\right]\na\varphi\dx\dt\\
			&-\tilde \sigma\int_0^T\hspace{-2mm}\int_{\Rd}\left[\left(1-\frac{\chi (z_n)_+}{\lvert\nabla\rho_{\eps,k}\rvert}\right)_+\na\rho_{\eps,k}-\left(1-\frac{\chi (z_n)_+}{\lvert\nabla\rho_{\eps,k}\rvert}\right)_+\na\rho_{\eps,k}\right]\na\varphi\dx\dt\\
			&-\eps\int_0^T\hspace{-2mm}\int_{\Rd}\left(\na\rho_{\eps,n}-\na\rho_{\eps,k}\right)\na\varphi\dx\dt-\eps\int_0^T\hspace{-2mm}\int_{\Rd}\left(\rho_{\eps,n}-\rho_{\eps,k}\right)\varphi\dx\dt\\
			=&-\tilde \sigma\int_0^T\hspace{-2mm}\int_{\Rd}\left[\left(1-\frac{\chi (z_n)_+}{\lvert\nabla\rho_{\eps,n}\rvert}\right)_+\na\rho_{\eps,n}-\left(1-\frac{\chi (z_n)_+}{\lvert\nabla\rho_{\eps,k}\rvert}\right)_+\na\rho_{\eps,k}\right]\na\varphi\dx\dt\\
			&-\tilde \sigma\int_0^T\hspace{-2mm}\int_{\Rd}\left[\left(1-\frac{\chi (z_n)_+}{\lvert\nabla\rho_{\eps,k}\rvert}\right)_+-\left(1-\frac{\chi (z_k)_+}{\lvert\nabla\rho_{\eps,k}\rvert}\right)_+\right]\na\rho_{\eps,k}\na\varphi\dx\dt\\
			& - \eps\int_0^T\hspace{-2mm}\int_{\Rd}\left(\na\rho_{\eps,n}-\na\rho_{\eps,k}\right)_+\na\varphi\dx\dt-\eps\int_0^T\hspace{-2mm}\int_{\Rd}\left(\rho_{\eps,n}-\rho_{\eps,k}\right)\varphi\dx\dt.
	\end{align*}}%
	We use the test function $\varphi=\rho_{\eps,n}-\rho_{\eps,k}$ and get
	 \begin{align}\label{continuity-computation}
			&\left( \int_{\Rd}  \rvert\rho_{\eps,n}-\rho_{\eps,k}\lvert^2\dx \right)(t)\nonumber\\
            \leq & -\tilde \sigma\int_0^T\hspace{-2mm}\int_{\Rd}\left[\left(1-\frac{\chi (z_n)_+}{\lvert\nabla\rho_{\eps,n}\rvert}\right)_+\na\rho_{\eps,n}-\left(1-\frac{\chi (z_n)_+}{\lvert\nabla\rho_{\eps,k}\rvert}\right)_+\na\rho_{\eps,k}\right]\na(\rho_{\eps,n}-\rho_{\eps,k})\dx\dt\nonumber \\
			&-\tilde \sigma\int_0^T\hspace{-2mm}\int_{\Rd}\left[\left(1-\frac{\chi (z_n)_+}{\lvert\nabla\rho_{\eps,k}\rvert}\right)_+\hspace{-2mm}-\left(1-\frac{\chi (z_k)_+}{\lvert\nabla\rho_{\eps,k}\rvert}\right)_+\right]\na\rho_{\eps,k}\cdot\na(\rho_{\eps,n}-\rho_{\eps,k})\dx\dt\nonumber\\
			 &- \eps\int_0^T\hspace{-2mm}\int_{\Rd}\lvert\na(\rho_{\eps,n}-\rho_{\eps,k})\rvert^2\dx\dt\nonumber\\
			\leq&-\tilde \sigma\int_0^T\hspace{-2mm}\int_{\Rd}\left[\left(1-\frac{\chi (z_n)_+}{\lvert\nabla\rho_{\eps,n}\rvert}\right)_+\na\rho_{\eps,n}-\left(1-\frac{\chi (z_n)_+}{\lvert\nabla\rho_{\eps,k}\rvert}\right)_+\na\rho_{\eps,k}\right]\na(\rho_{\eps,n}-\rho_{\eps,k})\dx\dt\nonumber\\
			&+\frac{\tilde \sigma^2}{2\eps}\int_0^T\hspace{-2mm}\int_{\Rd}\left[\left(1-\frac{\chi (z_n)_+}{\lvert\nabla\rho_{\eps,k}\rvert}\right)_+\hspace{-2mm}-\left(1-\frac{\chi (z_k)_+}{\lvert\nabla\rho_{\eps,k}\rvert}\right)_+\right]^2|\na\rho_{\eps,k}|^2\dx\dt \nonumber\\
			&- \frac{\eps}{2}\int_0^T\hspace{-2mm}\int_{\Rd}\lvert\na(\rho_{\eps,n}-\rho_{\eps,k})\rvert^2\dx\dt,
	\end{align}
 
using the Young's inequality in the second estimate.
Now come two key points in the argument. The first one is that, due to monotonicity, we have
\begin{align*}
	-\tilde \sigma\int_0^T\hspace{-2mm}\int_{\Rd}\left[\left(1-\frac{\chi (z_n)_+}{\lvert\nabla\rho_{\eps,n}\rvert}\right)_+\na\rho_{\eps,n}-\left(1-\frac{\chi (z_n)_+}{\lvert\nabla\rho_{\eps,k}\rvert}\right)_+\na\rho_{\eps,k}\right]\na(\rho_{\eps,n}-\rho_{\eps,k})\dx\dt\leq0.
\end{align*}
The second one is that $\left(1-\frac{\chi (z_n)_+}{\lvert\nabla\rho_{\eps,k}\rvert}\right)_+$ is Lipschitz continuous with respect to $(z_n)_+$. In fact, we have the following
\begin{align*}
	\left|\left(1-\frac{\chi (z_n)_+}{\lvert\nabla\rho_{\eps,k}\rvert}\right)_+-\left(1-\frac{\chi (z_k)_+}{\lvert\nabla\rho_{\eps,k}\rvert}\right)_+\right|\leq\frac{\chi}{|\na\rho_{\eps,k}|}|z_n-z_k|.
\end{align*}
Introducing this in \eqref{continuity-computation} yields:
\begin{align*}
			\left(\int_{\Rd}\rvert\rho_{\eps,n}-\rho_{\eps,k}\lvert^2\dx\right)(t)\leq&\;\frac{\chi \tilde \sigma^2}{2\eps}\int_0^T\hspace{-2mm}\int_{\Rd}\lvert (z_n)_+-(z_k)_+\rvert^2\dx\dt\\
   &-\frac{\eps}{2}\int_0^T\hspace{-2mm}\int_{\Rd}\lvert\na(\rho_{\eps,n}-\rho_{\eps,k})\rvert^2\dx\dt.
	\end{align*}%
	We conclude that $\{\na\rho_{\eps,n}\}$ is a Cauchy sequence in $L^2(0,T,L^2(\Rd)))$, and therefore, $\na\rho_{\eps,n}\rightarrow\na\rhoe$ a.e. $x\in[0,T]\times\Rd$. 
	
	Summarizing, we have shown that the map $\mathcal{T}$ is continuous, compact and constant when $\tilde \sigma=0$. Moreover, fixed points are uniformly bounded in $\tilde \sigma$. Then, Leray-Schauder fixed-point  theorem ensure the existence of $\rho_\eps$ such that  $\mathcal{T}(1,\rhoe)=\rhoe$, which is the solution to the nonlinear problem 
	\begin{align*}
		\int_0^T\int_{\Rd}\partial_t\rhoe\varphi\dx\dt=&-\int_0^T\int_{\Rd}\left(1-\frac{\chi{\rhoe}_+}{\lvert\nabla\rhoe\rvert}\right)_+\nabla\rhoe\nabla\varphi\dx\dt-\eps\int_0^T\int_{\Rd}\nabla\rhoe\nabla\varphi\dx\dt\nonumber\\
		&-\eps\int_0^T\int_{\Rd}\rhoe\varphi\dx\dt. 
	\end{align*}
To conclude the proof of existence of solutions to  \eqref{smooth.eq_bis}, it remains to show that  $\rhoe$ is nonnegative. For that, use $(\rho_{\eps})_-$ in the above weak formulation and obtain
	
	\begin{align*}
		\left(\int_{\Rd}(\rho_{\eps})_-^2\dx\right)(t)\leq\int_{\Rd}(\rho_{ { \rm in},\eps})_-^2\dx.
	\end{align*}
	Since $\rho_{ { \rm in},\eps}\geq0$ the last inequality implies ${\rho_{\eps}}_- \equiv 0$. This concludes the proof.

	

\section{The vanishing viscosity limit and proof of the main theorem}\label{sec_viscosity}
The aim of this section is to remove the viscosity terms $\eps\Delta\rhoe$ and $-\eps\rhoe$ from (\ref{smooth.eq_bis}).  As explained in the introduction, the main difficulty  is to pass to the limit in the nonlinear term $\left(1-\frac{\chi\rhoe}{|\na\rhoe|}\right)_+\na\rhoe$ and  show that
\begin{align*}
	\left(1-\frac{\chi\rhoe}{|\na\rhoe|}\right)_+\na\rhoe\rightharpoonup\left(1-\frac{\chi\rho}{|\na\rho|}\right)_+\na\rho\quad\mbox{in }L^2(0,T;\Rd).
\end{align*} 
From Proposition \ref{Proposition LP}  we know that
\begin{align*}
	\left\|\left(1-\frac{\chi\rhoe}{|\na\rhoe|}\right)_+\na\rhoe\right\|_{L^2(0,T;\Rd)}\leq C,
\end{align*}
which implies,
\begin{align*}
		\left(1-\frac{\chi\rhoe}{|\na\rhoe|}\right)_+\na\rhoe\rightharpoonup G\quad\mbox{in }L^2(0,T;\Rd).
\end{align*}
Moreover, Lemma \eqref{compactness.lemma} implies $\rhoe \to \rho$ a.e.. The crux is to show that $G=\left(1-\frac{\chi\rho}{|\na\rho|}\right)_+\na\rho$.

Using the Boltzmann relative entropy we will show that $\na\rhoe $ converges a.e. to $\na\rho $; in particular we will prove that
\begin{align*}
	\na\rhoe\rightarrow\na\rho\quad\mbox{a.e. in }\;{((0,T)\times\Rd)\backslash\{|\nabla\rho|\leq\chi\rho\}.}
\end{align*}
Let's briefly sketch how we get the last a.e. convergence; the  \emph{relative entropy} between two functions $u$ and $v$ is defined as:
\begin{align}\label{rel.entr}
	H_{{\rm rel}}[u\vert v]\equiv\int_{\Rd}\left( u\log\frac{u}{v} - u + v \right) \dx.
\end{align}
Notice that, for $u, v\geq 0$
\begin{align*}
	H_{{\rm rel}}[u\vert v](T)\geq 0,\quad\mbox{as}\quad
	r\log(r)-r+1\geq0, \quad\mbox{for all}\quad r\geq0,
\end{align*}
and $H[u|v]=0$ if and only if $u=v$. 
Consider $\rhoe$ and $\rhod$ two viscosity solutions to \eqref{smooth.eq_bis} with viscosity constants $\eps$ and $\delta$ respectively. A formal computation shows that 
\begin{align}\label{entropy comp 0}
	H[\rhoe|\rhod](T)-H[\rho_{{\rm in},\eps}|\rho_{{\rm in},\delta}]+\int_0^T\mathcal{D}[\rhoe|\rhod]\dt= \mbox{error terms.}
\end{align}
It turns out that the relative entropy production $\mathcal{D}[\rhoe|\rhod]$ is nonnegative, a direct consequence of the monotonicity property (\ref{monotonicity_prop}). 
The error terms originate from the viscosities $\varepsilon \Delta \rhoe$ and $\delta \Delta \rhod$. 
 As $\eps,\delta\rightarrow0$ these error terms vanish, as well as $H[\rho_{{\rm in},\eps}|\rho_{{\rm in},\delta}]\rightarrow0$. Since both $H$ and $\mathcal{D}$ are nonnegative, as $\delta , \varepsilon \to 0$, we must have 
\begin{align*}
	H[\rhoe|\rhod](T)\rightarrow0\quad\mbox{and} \quad \int_0^T\mathcal{D}[\rhoe|\rhod]\dt\rightarrow0.
\end{align*}
From the convergence 
\begin{align*}
\int_0^T\mathcal{D}[\rhoe|\rhod]\dt\rightarrow 0,
\end{align*}
we will obtain the almost everywhere convergence of $\na\rhoe$ to $\na\rho$ {in the nondegenerate region $\{|\nabla\rho|>\chi\rho\}$} and consequently will be able to identify  the limit function $G$ as $G = \left(1-\frac{\rho}{|\na\rho |}\right)_+\na\rho $. The rigorous proof of this argument is the content of sections \ref{sec:entropy_limit} and \ref{sec4.2}.

Then, in the last two subsections, we will prove the $L^p$ convergence of $\rho$ to initial data $\rhoin$ and the uniqueness.\subsection{Contractivity of the relative entropy}\label{sec:entropy_limit}

In what follows, we will make the argument above rigorous. We start by  testing  \eqref{smooth.eq_bis} for $\rhoe$ and for $\rhod$ respectively with test functions $\varphi_1=\log\left(\frac{\rhoe}{\rhod}\right)$ and $\varphi_2=\frac{\rhoe}{\rhod}-1$ (this will formally give $\frac{d}{dt} H[\rhoe|\rhod]$). Since neither $\varphi_1$ nor $\varphi_2$ are admissible test functions, we modify them accordingly, at the expense of some error terms. Specifically, our test functions are of the form
\begin{align*}
	\varphi_1=\log\left(\frac{\rhoe+\sigma}{\rhod+\sigma}\right),\;\;\mbox{and }\;\varphi_2=\frac{\rhoe+\sigma}{\rhod+\sigma}-1,
\end{align*}
with $\sigma>0$ a small constant. 
We point out that $\varphi_1$, $\varphi_2$ are indeed admissible test functions as they belong to $L^2(0,T; H^1(\R^d))$. For example,
$$
|\varphi_1|\leq \log(1+\rhoe/\sigma) + \log(1+\rhod/\sigma)
\leq (\rhoe + \rhod )/\sigma\in L^\infty(0,T; L^1\cap L^\infty(\R^d)).
$$
The integrability of $\varphi_2$, $\nabla\varphi_1$, $\nabla\varphi_2$ can be easily checked.

With this choice of $\varphi_1$ and $\varphi_2$ we obtain 
\begin{align}\label{equiv_entr}
	\int_{\Rd}&\left[ \varphi_1 \pa_t\rhoe -\varphi_2 \pa_t\rhod \right]\dx=\frac{d}{dt}\int_{\Rd}\left[(\rhoe+\sigma)\log{\left(\frac{\rhoe+\sigma}{\rhod+\sigma}\right)}-\rhoe+\rhod\right] \dx.
\end{align}
  This yields:
\begin{align}\label{entropy.derivative}
	 \int_0^T \int_{\Rd}\left[ \varphi_1 \pa_t\rhoe -\varphi_2 \pa_t\rhod \right]\dx\dt =& \int_0^T\hspace{-2mm}\int_{\Rd}\left[ \left(\log\frac{\rhoe+\sigma}{\rhod+\sigma}\right)\pa_t\rhoe +\left(1-\frac{\rhoe+\sigma}{\rhod+\sigma} \right)
	\;\pa_t\rhod \right] \dx\dt\\
	=&-\int_0^T\hspace{-2mm}\int_{\Rd}\left(1-\frac{\chi\rhoe}{\lvert\na\rhoe\rvert}\right)_+\na\rhoe\cdot \na\left(\log\frac{\rhoe+\sigma}{\rhod+\sigma}\right)
	\dx\dt\nonumber\\
	&+ \int_0^T\hspace{-2mm}\int_{\Rd}\left(1-\frac{\chi\rhod}{\lvert\na\rhod\rvert}\right)_+\na\rhod\cdot\na\left(\frac{\rhoe+\sigma}{\rhod+\sigma}\right)\dx\dt+E_{\eps,\delta,\sigma}\nonumber,
\end{align}
where $E_{\eps,\delta,\sigma}$ are error terms, which must vanish when $\eps,\delta,\sigma\rightarrow0$. Notice that \eqref{equiv_entr} is an approximation of $\int_0^T\frac{d}{\dt}H_{{\rm rel}}[\rhoe\vert\rhod] \;dt$. We will show 
\begin{align*}
	\lim\limits_{\eps, \delta\rightarrow0}\int_0^T
	\left(
	\frac{d}{\dt}H_{{\rm rel}}[\rhoe\vert\rhod]
	{ + \mathcal{D}[\rhoe\vert\rhod]}
	\right)\dt 
	\leq 0,
\end{align*} 
to find that, in fact, $\lim\limits_{\eps, \delta\rightarrow0}H_{{\rm rel}}[\rhoe\vert\rhod](T)=0$
and $\lim\limits_{\eps, \delta\rightarrow0}\int_0^T \mathcal{D}[\rhoe\vert\rhod]\dt = 0$.
These limits will provide the necessary information about the convergence of $\nabla\rhoe$. 

Let us first analyze the error terms:
\begin{subequations}
\begin{align}
E_{\eps,\delta,\sigma}=&-\eps\int_0^T\hspace{-2mm}\IRd\rhoe\log\frac{\rhoe+\sigma}{\rhod+\sigma}\dx\dt+\delta\int_0^T\hspace{-2mm}\IRd\rhod{\left(\frac{\rhoe+\sigma}{\rhod+\sigma} - 1\right)}\dx\dt\label{error.0}\\
&-\int_0^T\hspace{-2mm}\IRd\left[\eps\na\rhoe\na\left(\log\frac{\rhoe+\sigma}{\rhod+\sigma}\right)-\delta\na\rhod\na\left(\frac{\rhoe+\sigma}{\rhod+\sigma}\right)\right]\dx\dt.\label{error.1}
\end{align}
\end{subequations}
{We notice first that
	\begin{align}\label{elem.log}
		\left|\sigma\log\left(
		\frac{a+\sigma}{b+\sigma}
		\right) \right|
		= \left|\sigma\log\left(1 + \frac{a}{\sigma}\right)
		-\sigma\log\left(1 + \frac{b}{\sigma}\right)\right|
		\leq a+b\quad\forall a,b\geq 0.
	\end{align}
}
For the first term in $E_{\eps,\delta}$, assuming without loss of generality $\eps<\delta<\frac{1}{2}$ we compute {(remember that $H_{rel}[\rhoe\vert\rhod]\geq 0$)}
\begin{align}\label{error.calculus.1}
\eqref{error.0}
=&-\eps H_{rel}[\rhoe\vert\rhod]+\eps\sigma\int_0^T\hspace{-2mm}\IRd\log\frac{\rhoe+\sigma}{\rhod+\sigma}\dx\dt-\eps\int_0^T\hspace{-2mm}\IRd(\rhoe-\rhod)\dx\dt\nonumber\\
&+\delta\int_0^T\hspace{-2mm}\IRd\rhod\frac{\rhoe - \rhod}{\rhod+\sigma}\dx\dt\nonumber\\
\leq&\; \eps\int_0^T\hspace{-2mm}\IRd (\rhoe+\rhod) \dx\dt
+\eps\int_0^T\hspace{-2mm}\IRd\rhod \dx\dt
 +\delta\int_0^T\hspace{-2mm}\IRd \rhoe \dx\dt\nonumber\\
\leq& 2(\eps+\delta)\int_0^T\hspace{-2mm}\IRd
\left(\rhoe+\rhod\right)\dx\dt\nonumber\\
\leq& {C(\eps + \delta)},
\end{align} where $C>0$ only depends on the $L^1$ norm of the initial data.  From now on, $C$ will be a positive universal constant that can change from line to line. %

For expression \eqref{error.1}, we compute
\begin{align}\label{error.calculus.2}
\mbox{\eqref{error.1}}=&-\eps\int_0^T\hspace{-2mm}\IRd\frac{\vert\na\rhoe\vert^2}{\rhoe+\sigma}\dx\dt+(\eps+\delta)\int_0^T\hspace{-2mm}\IRd\frac{\na\rhoe\na\rhod}{\rhod+\sigma}\dx\dt\nonumber\\
&-\delta\int_0^T\hspace{-2mm}\IRd(\rhoe+\sigma)\left\vert\frac{\na\rhod}{\rhod+\sigma}\right\vert^2\dx\dt\nonumber\\
\leq&\frac{(\eps+\delta)}{2}\int_0^T\hspace{-2mm}\IRd\frac{\vert\na\rhoe\vert^2}{\rhoe+\sigma}\dx\dt+\frac{(\eps+\delta)}{2}\int_0^T\hspace{-2mm}\IRd(\rhoe+\sigma)\left\vert\frac{\na\rhod}{\rhod+\sigma}\right\vert^2\dx\dt\nonumber\\
&-\delta\int_0^T\hspace{-2mm}\IRd(\rhoe+\sigma)\left\vert\frac{\na\rhod}{\rhod+\sigma}\right\vert^2\dx\dt\nonumber\\
\leq&\frac{(\eps+\delta)}{2}\int_0^T\hspace{-2mm}\IRd\frac{\vert\na\rhoe\vert^2}{\rhoe}\dx\dt\nonumber\\
\leq& {C(\eps + \delta)},
\end{align}
using Lemma \ref{grad.log.estimate} in the last inequality.
Summarizing, by \eqref{error.calculus.1} and \eqref{error.calculus.2} we have the following estimate for the error term
\begin{align}\label{error_total_estimate}
	E_{\eps,\delta}
	{\leq C(\eps + \delta)}.
\end{align}
Now, let us see what happens with the remaining terms in \eqref{entropy.derivative}. These are
\begin{align}\label{main term entropy derivative}
& {I_3} := -\int_0^T\hspace{-2mm}\IRd\left(1-\frac{\chi\rhoe}{\lvert\na\rhoe\rvert}\right)_+\na\rhoe \na\left(\log\frac{\rhoe+\sigma}{\rhod+\sigma}\right)
\dx\dt,\\
&{I_4} := \int_0^T\hspace{-2mm}\IRd\left(1-\frac{\chi\rhod}{\lvert\na\rhod\rvert}\right)_+
\na\rhod\na\left(\frac{\rhoe+\sigma}{\rhod+\sigma}\right)
\dx\dt.
	\end{align}

We can reformulate {$I_3$} as 
\begin{align*}
{I_3}=&-\int_0^T\hspace{-2mm}\IRd\left(1-\frac{\chi(\rhoe+\sigma)}{\lvert\na\rhoe\rvert}\right)_+\na\rhoe \na\left(\log\frac{\rhoe+\sigma}{\rhod+\sigma}\right)
\dx\dt\\
&-\int_0^T\hspace{-2mm}\IRd\left[\left(1-\frac{\chi\rhoe}{\lvert\na\rhoe\rvert}\right)_+-\left(1-\frac{\chi(\rhoe+\sigma)}{\lvert\na\rhoe\rvert}\right)_+\right]\na\rhoe \na\left(\log\frac{\rhoe+\sigma}{\rhod+\sigma}\right)
\dx\dt .
\end{align*}
{We rewrite $I_4$ in a similar way. It follows}
\begin{align*}
 {I_3 + I_4}=& \; -\int_0^T\hspace{-2mm}\IRd\left(1-\frac{\chi(\rhoe+\sigma)}{\lvert\na\rhoe\rvert}\right)_+\na\rhoe \na\left(\log\frac{\rhoe+\sigma}{\rhod+\sigma}\right)
 \dx\dt\\
&+\int_0^T\hspace{-2mm}\IRd\left(1-\frac{\chi(\rhod+\sigma)}{\lvert\na\rhod\rvert}\right)_+\na\rhod\na\left(\frac{\rhoe+\sigma}{\rhod+\sigma}\right)\dx\dt + R_1 + R_2,
\end{align*}
where
\begin{align*}
& R_1: =-\int_0^T\hspace{-2mm}\IRd\left[\left(1-\frac{\chi\rhoe}{\lvert\na\rhoe\rvert}\right)_+-\left(1-\frac{\chi(\rhoe+\sigma)}{\lvert\na\rhoe\rvert}\right)_+\right]\na\rhoe \na\left(\log\frac{\rhoe+\sigma}{\rhod+\sigma}\right)
\dx\dt,\\
& R_2:=\int_0^T\hspace{-2mm}\IRd\left[\left(1-\frac{\chi\rhod}{\lvert\na\rhod\rvert}\right)_+-\left(1-\frac{\chi(\rhod+\sigma)}{\lvert\na\rhod\rvert}\right)_+\right]\na\rhod\na\left(\frac{\rhoe+\sigma}{\rhod+\sigma}\right)\dx\dt.
\end{align*}
We recall that 
\begin{align}\label{lipschitz condition}
	\left|\left(1-\frac{\chi\rhoe}{\lvert\na\rhoe\rvert}\right)_+-\left(1-\frac{\chi(\rhoe+\sigma)}{\lvert\na\rhoe\rvert}\right)_+\right|\leq\frac{\chi\sigma}{|\na\rhoe|},
\end{align}
and also that 
\begin{align*}
	0\leq\left[\left(1-\frac{\chi\rhoe}{\lvert\na\rhoe\rvert}\right)_+-\left(1-\frac{\chi(\rhoe+\sigma)}{\lvert\na\rhoe\rvert}\right)_+\right].
\end{align*}
Similar bounds hold in $R_2$ for the term with $\rhod$.
Using this, we estimate $R_1$ as follows:
\begin{align*}
R_1\leq&\int_0^T\hspace{-2mm}\IRd\left[\left(1-\frac{\chi\rhoe}{\lvert\na\rhoe\rvert}\right)_+-\left(1-\frac{\chi(\rhoe+\sigma)}{\lvert\na\rhoe\rvert}\right)_+\right]\na\rhoe\frac{\na\rhod}{\rhod+\sigma}\dx\dt\\
	\leq& \; \chi\int_0^T\hspace{-2mm}\IRd\frac{\sigma|\na\rhod|}{\rhod+\sigma}\dx\dt.
\end{align*}
{For fixed $\eps,\delta>0$ the expression above tends to zero as $\sigma\to 0$. Indeed, consider the sequence of functions
	$F_\sigma = \frac{\sigma|\na\rhod|}{\rhod+\sigma}$. It holds
	$F_\sigma \leq \overline{F}\equiv |\na\rhod| = 2\sqrt{\rhod}|\nabla\sqrt{\rhod}|\in L^2(0,T; L^1(\R^d))$. Furthermore $F_\sigma\to 0$ a.e.~in $\{\rhod>0\}$ as $\sigma\to 0$, while $F_\sigma = |\na\rhod| = 2\sqrt{\rhod}|\nabla\sqrt{\rhod} = 0$ a.e.~in $\{\rhod=0\}$. Hence by dominated convergence we obtain that $F_\sigma\to 0$ strongly in $L^1(0,T; L^1(\R^d))$.
	This means
	\begin{align}\label{limsup.R1}
		\limsup_{\sigma\to 0}R_1\leq 0\quad\forall\eps,\delta>0.
	\end{align}
}
We estimate $R_2$ in a similar way:
\begin{align*}
R_2\leq&\int_0^T\hspace{-2mm}\IRd\left[\left(1-\frac{\chi\rhod}{\lvert\na\rhod\rvert}\right)_+-\left(1-\frac{\chi(\rhod+\sigma)}{\lvert\na\rhod\rvert}\right)_+\right]\frac{\na\rhod\na\rhoe}{\rhod+\sigma}\dx\dt
{\leq\int_0^T\hspace{-2mm}\IRd G_\sigma\dx\dt , }
\end{align*}
{with 
$$G_\sigma\equiv \left[\left(1-\frac{\chi\rhod}{\lvert\na\rhod\rvert}\right)_+-\left(1-\frac{\chi(\rhod+\sigma)}{\lvert\na\rhod\rvert}\right)_+\right]\frac{|\na\rhod| |\na\rhoe|}{\rhod+\sigma}.$$
We first observe that (thanks to \eqref{lipschitz condition})
$$ 0\leq
G_\sigma \leq \chi\sigma\frac{|\na\rhoe|}{\rhod+\sigma}\leq 
\chi |\na\rhoe|\equiv \overline{G}
$$
and $\overline{G}$ is a $\sigma-$independent integrable function. The same estimate also tells us that $G_\sigma\to 0$ a.e.~on $\{\rhod>0\}$ as $\sigma\to 0$. On the other hand, on $\{\rhod=0\}$ $G_\sigma$ equals a bounded factor (the one between square brackets) times $|\nabla\rhoe|$ (which is a.e.~finite because it is integrable) times $|\nabla\rhod|$, which vanishes a.e.~in $\{\rhod=0\}$ as previously noticed. Again, via the dominated convergence theorem we infer that $G_\sigma\to 0$ strongly in $L^1(0,T; L^1(\R^d))$ as $\sigma\to 0$, for fixed (arbitrary) $\eps$, $\delta>0$. Hence we obtain
\begin{align}\label{limsup.R2}
\limsup_{\sigma\to 0}R_2\leq 0\quad\forall\eps,\delta>0.
\end{align}
}



%
{Straightforward computations lead to}
\begin{align*}
{I_3 + I_4}
	=&-\int_0^T\hspace{-2mm}\IRd \frac{\rhoe+\sigma}{2}\left[\left(1-\frac{\chi(\rhoe+\sigma)}{\lvert\na\rhoe\rvert}\right)_+\hspace{-2mm}-\left(1-\frac{\chi(\rhod+\sigma)}{\lvert\na\rhod\rvert}\right)_+\right]\times\\
	&\times\left(\lvert\na\log(\rhoe+\sigma)\rvert^2-\lvert\na\log(\rhod+\sigma)\rvert^2\right)\dx\dt\\
	&-\int_0^T\hspace{-2mm}\IRd \frac{\rhoe+\sigma}{2}\left\vert\na\log\left(\frac{\rhoe+\sigma}{\rhod+\sigma}\right)\right\vert^2\left[\left(1-\frac{\chi(\rhoe+\sigma)}{\lvert\na\rhoe\rvert}\right)_++\left(1-\frac{\chi(\rhod+\sigma)}{\lvert\na\rhod\rvert}\right)_+\right]\dx\dt\\ 
 &{+R_1 + R_2}\\
	\leq&\;-\int_0^T\hspace{-2mm}\IRd \frac{\rhoe+\sigma}{2}\left[\left(1-\frac{\chi(\rhoe+\sigma)}{|\na\rhoe|}\right)_+-\left(1-\frac{\chi(\rhod+\sigma)}{|\na\rhod|}\right)_+\right]^2\dx\dt\\
	&-\int_0^T\hspace{-2mm}\IRd \frac{\rhoe+\sigma}{2}\left\vert\na\log\left(\frac{\rhoe+\sigma}{\rhod+\sigma}\right)\right\vert^2\left[\left(1-\frac{\chi(\rhoe+\sigma)}{\lvert\na\rhoe\rvert}\right)_++\left(1-\frac{\chi(\rhod+\sigma)}{\lvert\na\rhod\rvert}\right)_+\right]\dx\dt\\
	&+ R_1+R_2,
\end{align*}
where, in the second inequality, we have used that
\begin{align*}
	\left[\left(1-\frac{\chi}{r}\right)_+-\left(1-\frac{\chi}{s}\right)_+\right]^2\leq C\left[\left(1-\frac{\chi}{r}\right)_+-\left(1-\frac{\chi}{s}\right)_+\right]\left(r^2-s^2\right), \quad\mbox{for any}\quad r, s>0.
\end{align*}
\normalsize
Therefore, \eqref{equiv_entr} and \eqref{entropy.derivative} give
\begin{align}\label{entr.comp.99}
&\left(\int_{\Rd}\left[(\rhoe+\sigma)\log{\left(\frac{\rhoe+\sigma}{\rhod+\sigma}\right)}-\rhoe+\rhod\right]\dx\right)(T)\nonumber \\
&\qquad {- \int_{\Rd}\left[(\rho_{in,\eps}+\sigma)\log{\left(\frac{\rho_{in,\eps}+\sigma}{\rho_{in,\delta}+\sigma}\right)}-\rho_{in,\eps}+\rho_{in,\delta}\right]\dx}
 \nonumber\\
&={\IRd (\rhod(T) - \rho_{in,\delta})\phi_R\dx}\nonumber\\
&\;\;\; -\frac{1}{2}\int_0^T\hspace{-2mm}\IRd(\rhoe+\sigma)\left[\left(1-\frac{\chi(\rhoe+\sigma)}{|\na\rhoe|}\right)_+-\left(1-\frac{\chi(\rhod+\sigma)}{|\na\rhod|}\right)_+\right]^2\dx\dt\nonumber\\
	&\;\;\;-\frac{1}{2}\int_0^T\hspace{-2mm}\IRd(\rhoe+\sigma)\left\vert\na\log\left(\frac{\rhoe+\sigma}{\rhod+\sigma}\right)\right\vert^2\left[\left(1-\frac{\chi(\rhoe+\sigma)}{\lvert\na\rhoe\rvert}\right)_+\hspace{-2.5mm}+\left(1-\frac{\chi(\rhod+\sigma)}{\lvert\na\rhod\rvert}\right)_+\right]\dx\dt\nonumber\\
	&\;\;\;+R_1+R_2 + E_{\eps,\delta,\sigma},
\end{align}
where 
{$R_1+R_2\to 0$ as $\sigma\to 0$, for every $\eps,\delta>0$, while $E_{\eps,\delta,\sigma}\leq C(\eps+\delta)$ for every $\sigma>0$.}

Now, we are going to find a bound for ${\IRd (\rhod(T) - \rho_{in,\delta})\dx}$. This is done in a rigorous way by employing a cutoff $\phi_R = \phi_R(x)$ as test function in \eqref{smooth.eq_bis}. The cutoff $\phi_R$ has the following properties:
$$
\phi_R\in C_c^\infty(\R^d),~~ 0\leq\phi_R\leq 1\mbox{ in }\R^d,~~ \phi_R=1\mbox{ on }B_R,~~\phi_R=0\mbox{ on }\R^d\backslash B_{2R},~~|\nabla\phi_R|\leq CR^{-1}\mbox{ on }\R^d.
$$
We obtain
\begin{align}\label{entr.comp.100}
\left(\int_{\Rd}\rhod{\phi_R}\dx\right)(T)=&\int_{\Rd}\rho_{{\rm in},\delta}{\phi_R}\dx-\int_0^T\int_{\Rd}\left(1-\frac{\chi\rhod}{\lvert\nabla\rhod\rvert}\right)_+\nabla\rhod\nabla\phi_R\dx\dt\nonumber\\
	&-\delta\int_0^T\int_{\Rd}\nabla\rhod\nabla\phi_R\dx\dt-\delta\int_0^T\int_{\Rd}\rhod\phi_R\dx\dt.
\end{align}
Since $|\na\phi_R|\leq\frac{C}{R}$ we have that 
\begin{align*}
\left\lvert\int_0^T\int_{\Rd}\left(1-\frac{\chi\rhod}{\lvert\nabla\rhod\rvert}\right)_+\nabla\rhod\nabla\phi_R\dx\dt\right\rvert \leq\frac{C}{R}\|\na\sqrt{\rhod}\|_2\|\rhod\|_1^{1/2}
{\leq \frac{C}{R}}.
\end{align*}
{A similar estimate can be obtained for}
the viscosity terms in \eqref{entr.comp.100}, and therefore we get that 
\begin{align*}
{\left|\left(\int_{\Rd}\rhod{\phi_R}\dx\right)(T) - \int_{\Rd}\rho_{{\rm in},\delta}{\phi_R}\dx\right|\leq  \frac{C}{R} + C\delta}.
\end{align*}
{Taking the limit $R\to \infty$ in the above estimate leads to}
\begin{align*}
{ \left|\left(\int_{\Rd}\rhod \dx\right)(T) - \int_{\Rd}\rho_{{\rm in},\delta} \dx\right|\leq C\delta}.
\end{align*}
Hence, taking the (superior) limit $\sigma\rightarrow0$ in \eqref{entr.comp.99} and applying Fatou's Lemma yield
\begin{align*}
&\limsup_{\sigma\to 0}
\int_0^T\hspace{-2mm}\IRd(\rhoe+\sigma)\left[\left(1-\frac{\chi(\rhoe+\sigma)}{|\na\rhoe|}\right)_+ \hspace{-4mm}-\left(1-\frac{\chi(\rhod+\sigma)}{|\na\rhod|}\right)_+\right]^2\dx\dt\leq 
H_{rel}[\rho_{in,\eps}\vert \rho_{in,\delta}] + C(\eps + \delta),\\
&\limsup_{\sigma\rightarrow0}\int_0^T\hspace{-2mm}\IRd(\rhoe+\sigma)\left\vert\na\log\left(\frac{\rhoe+\sigma}{\rhod+\sigma}\right)\right\vert^2 \cdot \nonumber\\
& \;\;\;\;\;\;\;\;\;\;\;\; \cdot \left[\left(1-\frac{\chi(\rhoe+\sigma)}{\lvert\na\rhoe\rvert}\right)_++\left(1-\frac{\chi(\rhod+\sigma)}{\lvert\na\rhod\rvert}\right)_+\right]\dx\dt
\leq 
H_{rel}[\rho_{in,\eps}\vert \rho_{in,\delta}] + C(\eps + \delta).
\end{align*}
We need now to compute the a.e.~limit as $\sigma\to 0$ of the integrands in the above relations. It holds:
\begin{align*}
&\mbox{On }\{\rhoe>0\}:~~
\left(1-\frac{\chi(\rhoe+\sigma)}{|\na\rhoe|}\right)_+
\to \left(1-\frac{\chi\rhoe}{|\na\rhoe|}\right)_+~~\mbox{ as }\sigma\to 0,\\
&\mbox{on }\{\rhoe=0\}:~~
\left(1-\frac{\chi(\rhoe+\sigma)}{|\na\rhoe|}\right)_+
=\left(1-\frac{\chi\sigma}{|\na\rhoe|}\right)_+ = 0,
\end{align*}
since $\nabla\rhoe = 0$ a.e.~on $\{\rhoe=0\}$. This means that
\begin{align*}
\left(1-\frac{\chi(\rhoe+\sigma)}{|\na\rhoe|}\right)_+
\to \left(1-\frac{\chi\rhoe}{|\na\rhoe|}\right)_+{\bf 1}_{\{\rhoe>0\}}\quad \mbox{ a.e.~in $\R^d\times (0,T)$ as }\sigma\to 0.
\end{align*}
Let us now consider
\begin{align*}
\na\log\left(\frac{\rhoe+\sigma}{\rhod+\sigma}\right) = 
\frac{\na\rhoe}{\rhoe + \sigma} - \frac{\na\rhod}{\rhod + \sigma}.
\end{align*}
It holds
\begin{align*}
&\mbox{On }\{\rhoe>0\}:~~
\frac{\na\rhoe}{\rhoe + \sigma}\to 
\frac{\na\rhoe}{\rhoe}\quad\mbox{ as }\sigma\to 0,\\
&\mbox{on }\{\rhoe = 0\}:~~
\frac{\na\rhoe}{\rhoe + \sigma} = \frac{\na\rhoe}{\sigma} = 0,
\end{align*}
again because $\nabla\rhoe = 0$ a.e.~on $\{\rhoe=0\}$. It follows
\begin{align*}
\na\log\left(\frac{\rhoe+\sigma}{\rhod+\sigma}\right)\to 
\frac{\na\rhoe}{\rhoe}{\bf 1}_{\{\rhoe>0\}} - 
\frac{\na\rhod}{\rhod}{\bf 1}_{\{\rhod>0\}}\quad\mbox{ as }\sigma\to 0.
\end{align*}
Hence we get
\begin{align*}
&\int_0^T\hspace{-2mm}\IRd \rhoe\left[\left(1-\frac{\chi\rhoe}{|\na\rhoe|}\right)_+{\bf 1}_{\{\rhoe>0\}} -\left(1-\frac{\rhod}{|\na\rhod|}\right)_+{\bf 1}_{\{\rhod>0\}}
	\right]^2\dx\dt\leq H_{rel}[\rho_{in,\eps}\vert \rho_{in,\delta}] + C(\eps + \delta),\\
&\int_0^T\hspace{-2mm}\IRd\rhoe\left\vert
	\frac{\na\rhoe}{\rhoe} {\bf 1}_{\{ \rhoe>0 \}} - 
	\frac{\na\rhod}{\rhod} {\bf 1}_{\{\rhod>0\}}
	\right\vert^2 \cdot \nonumber\\
	& \;\;\;\;\;\;\;\;\;\;\;\; \cdot 
	\left[
	\left(1-\frac{\chi\rhoe}{\lvert\na\rhoe\rvert}\right)_+{\bf 1}_{\{\rhoe>0\}}
	+\left(1-\frac{\chi\rhod}{\lvert\na\rhod\rvert}\right)_+{\bf 1}_{\{\rhod>0\}}
	\right]\dx\dt\leq H_{rel}[\rho_{in,\eps}\vert \rho_{in,\delta}] + C(\eps + \delta).
\end{align*}
Therefore by taking the limit $\eps,\delta\to 0$ in the above estimates we obtain
\begin{subequations}
\begin{align}
&\lim_{\epsilon,\delta\rightarrow0}
\int_0^T\hspace{-2mm}\IRd \rhoe\left[\left(1-\frac{\chi\rhoe}{|\na\rhoe|}\right)_+{\bf 1}_{\{\rhoe>0\}} -\left(1-\frac{\rhod}{|\na\rhod|}\right)_+{\bf 1}_{\{\rhod>0\}}
\right]^2\dx\dt=0,\label{going_to_0_1}\\
&\lim_{\epsilon,\delta\rightarrow 0}
\int_0^T\hspace{-2mm}\IRd\rhoe\left\vert
\frac{\na\rhoe}{\rhoe} {\bf 1}_{\{ \rhoe>0 \}} - 
\frac{\na\rhod}{\rhod} {\bf 1}_{\{\rhod>0\}}
\right\vert^2 \cdot \nonumber\\
& \;\;\;\;\;\;\;\;\;\;\;\; \cdot 
\left[
\left(1-\frac{\chi\rhoe}{\lvert\na\rhoe\rvert}\right)_+{\bf 1}_{\{\rhoe>0\}}
+\left(1-\frac{\chi\rhod}{\lvert\na\rhod\rvert}\right)_+{\bf 1}_{\{\rhod>0\}}
\right]\dx\dt=0.\label{going_to_0_2}
\end{align}
\end{subequations}

\subsection{Identification of the limit of the flux} \label{sec4.2}
Our goal is to show that
\begin{align*}
\left(1-\frac{\chi\rhoe}{\lvert\na\rhoe\rvert}\right)_+\na\rhoe\rightarrow\left(1-\frac{\chi\rho}{\lvert\na\rho\rvert}\right)_+\na\rho\quad\mbox{a.e. in } (0,T)\times \Rd.
\end{align*}
The above a.e.~convergence relation will then allow us to identify the weak $L^2$ limit of the term $\left(1-\frac{\chi\rhoe}{\lvert\na\rhoe\rvert}\right)_+\na\rhoe$.
By Lemma \ref{compactness.lemma}, we know that 
\begin{align*}
	\rhoe\rightarrow\rho\quad \mbox{strongly in } L^2((0,T)\times\Rd).
\end{align*}
Given mass conservation it follows that $\rhoe\to\rho$ strongly in $L^p((0,T)\times\Rd)$ for every $p>1$. Putting this fact together with the uniform moment bounds for $\rhoe$ we deduce that
\begin{align*}
\rhoe\rightarrow\rho\quad \mbox{strongly in } L^1((0,T)\times\Rd).
\end{align*}
From the above convergence property and \eqref{going_to_0_1} one easily obtains
\begin{align*}
\lim_{\epsilon,\delta\rightarrow0}
\int_0^T\hspace{-2mm}\IRd \rho\left[\left(1-\frac{\chi\rhoe}{|\na\rhoe|}\right)_+{\bf 1}_{\{\rhoe>0\}} -\left(1-\frac{\rhod}{|\na\rhod|}\right)_+{\bf 1}_{\{\rhod>0\}}
\right]^2\dx\dt=0
\end{align*}
which means that the sequence $\left(\left(1-\frac{\chi\rhoe}{|\na\rhoe|}\right)_+{\bf 1}_{\{\rhoe>0\}}\right)_{\eps>0}$ is Cauchy in $L^2(\R^d\times (0,T),\rho\dx\dt)$. Therefore, a (unique) function $\Xi\in L^2(\R^d\times (0,T))$ exists such that 
\begin{align}\label{lim.Xi.L2}
\left(1-\frac{\chi\rhoe}{|\na\rhoe|}\right)_+{\bf 1}_{\{\rhoe>0\}}\rightarrow \Xi\quad\mbox{strongly in }L^2(\R^d\times (0,T),\rho\dx\dt).
\end{align}
In particular, up to subsequences it holds
\begin{align}\label{lim.Xi.ae}
\left(1-\frac{\chi\rhoe}{|\na\rhoe|}\right)_+{\bf 1}_{\{\rhoe>0\}}\rightarrow \Xi
\quad\mbox{a.e.~in }[0,T]\times\Rd\setminus\{\rho=0\}.
\end{align}
And on the other hand, from \eqref{going_to_0_2} we deduce the a.e.~convergence (up to subsequences) of the integrand:
\begin{align*}
\lim_{\epsilon,\delta\rightarrow 0}
\rhoe\left\vert
\frac{\na\rhoe}{\rhoe} {\bf 1}_{\{ \rhoe>0 \}} - 
\frac{\na\rhod}{\rhod} {\bf 1}_{\{\rhod>0\}}
\right\vert^2   
\left[
\left(1-\frac{\chi\rhoe}{\lvert\na\rhoe\rvert}\right)_+{\bf 1}_{\{\rhoe>0\}}
+\left(1-\frac{\chi\rhod}{\lvert\na\rhod\rvert}\right)_+{\bf 1}_{\{\rhod>0\}}
\right]=0,\\ \mbox{a.e.~in }\R^d\times (0,T)
\end{align*}
therefore
\begin{align*}
\lim_{\epsilon,\delta\rightarrow 0}\left\vert
\frac{\na\rhoe}{\rhoe} {\bf 1}_{\{ \rhoe>0 \}} - 
\frac{\na\rhod}{\rhod} {\bf 1}_{\{\rhod>0\}}
\right\vert = 0\quad\mbox{a.e.~on }[0,T]\times\Rd\setminus\{\{\rho=0\}\cup\{\Xi=0\}\}.
\end{align*}
This means that a.e.~in $\R^d\times (0,T)$ the sequence $(\frac{\na\rhoe}{\rhoe} {\bf 1}_{\{ \rhoe>0 \}})_{\eps>0}$ is Cauchy and therefore convergent towards some function $w$:
\begin{align}\label{lim.nalogrho.ae}
\frac{\na\rhoe}{\rhoe}{\bf 1}_{\{ \rhoe>0 \}}\rightarrow w\quad\mbox{a.e. in }[0,T]\times\Rd\setminus\{
\{\rho=0\}\cup\{\Xi=0\}\}.
\end{align}
Notice in particular that the a.e.~convergence of ${\bf 1}_{\{ \rhoe>0 \}}\na\rhoe$ in $\{\rho>0,\Xi>0\}$ follows (since $\rhoe$ is a.e.~convergent). Given that ${\bf 1}_{\{ \rhoe=0 \}}\na\rhoe = 0$ a.e.~in $\R^d\times (0,T)$, we infer the a.e.~convergence of $\na\rhoe$ on $\{\rho>0,\,\Xi>0\}$. Furthermore, from \eqref{lim.Xi.ae} and the a.e.~convergence of $\rhoe$ we deduce that 
\begin{align}\label{limsup.narho.ae}
\limsup_{\eps\to 0}|\na\rhoe|\leq \chi\rho\quad\mbox{a.e.~on }\{\rho>0,\,\Xi=0\} .
\end{align}
In particular the sequence $(\na\rhoe(x,t))_{\eps>0}$ is bounded for a.e.~$(x,t)\in\R^d\times (0,T)\cap\{\rho>0\}$ (with a bound dependent on $(x,t)$). 

Given the previous two a.e.~convergence relations we deduce
\begin{align*}
\rhoe^3\frac{\left\vert\na\rhoe\right\vert}{\rhoe}{\bf 1}_{\{ \rhoe>0 \}}
\left(\frac{\left\vert\na\rhoe\right\vert}{\rhoe}{\bf 1}_{\{ \rhoe>0 \}} +\chi\right)\left(1-\frac{\chi\rhoe}{\vert\na\rhoe\vert}\right)_+{\bf 1}_{\{ \rhoe>0 \}}\frac{\na\rhoe}{\rhoe}{\bf 1}_{\{ \rhoe>0 \}}\rightarrow \overline{w},
\end{align*}
almost everywhere in $[0,T]\times\Rd\setminus\{\rho=0\}$, or, which is equivalent,
\begin{align}\label{cnv.tmp.1}
{\bf 1}_{\{ \rhoe>0 \}}\left(\vert\na\rhoe\vert^2-\rhoe^2\chi^2\right)_+\na\rhoe\rightarrow \overline{w},
\end{align}
almost everywhere in $[0,T]\times\Rd\setminus\{\rho=0\}$.
In fact, while there is clearly convergence a.e.~on $\{\rho>0,\,\Xi>0\}$, on $\{\rho>0,\,\Xi=0\}$ the term $\na\rhoe$ is bounded thanks to \eqref{limsup.narho.ae} while $\left(1-\frac{\chi\rhoe}{\vert\na\rhoe\vert}\right)_+{\bf 1}_{\{ \rhoe>0 \}}$ tends to zero.

Notice that ${\bf 1}_{\{ \rhoe>0 \}}$ can be eliminated from \eqref{cnv.tmp.1} since $\left(\vert\na\rhoe\vert^2-\rhoe^2\chi^2\right)_+\na\rhoe = 0$ on $\{\rhoe=0\}$.
Moreover, as $\rhoe\rightarrow\rho$ a.e. in $(0,T)\times\Rd$ and the sequence $(\na\rhoe(x,t))_{\eps>0}$ is bounded for a.e.~$(x,t)\in\R^d\times (0,T)\cap\{\rho>0\}$, we infer 
\begin{align*}
	\left(\vert\na\rhoe\vert^2-\rho^2\chi^2\right)_+\na\rhoe\rightarrow \overline{w} ,
\end{align*}
almost everywhere in $[0,T]\times\Rd\setminus\{\rho=0\}$. 

Define $Q_{R,\tau,T}=\{[\tau,T]\times B_R \cap\{\rho>0\}\}$. 
Egorov's and Lusin's Theorems imply that, for every $\eta>0$, there exists $Q_{R,\tau,T}^\eta\subset Q_{R,\tau,T}$ compact such that $|Q_{R,\tau,T}\backslash Q_{R,\tau,T}^\eta|<\eta$, 
\begin{align*}
\left(\vert\na\rhoe\vert^2-\rho^2\chi^2\right)_+\na\rhoe\in C^0(Q_{R,\tau,T}^\eta),\\
\left(\vert\na\rhoe\vert^2-\rho^2\chi^2\right)_+\na\rhoe \to \overline{w}\quad \mbox{uniformly in }Q_{R,\tau,T}^\eta .
\end{align*}

Then, as $\na\rhoe\rightharpoonup\na\rho$ weakly in {$L^p(\R^d\times (0,T))$} we have that 
\begin{align*}
	&\int_{Q_{R,\tau,T}^\eta}\left[\left(\vert\na\rhoe\vert^2-\rho^2\chi^2\right)_+\na\rhoe-\left(\vert\na\rho\vert^2-\rho^2\chi^2\right)_+\na\rho\right](\na\rhoe-\na\rho)\dx\dt\\
	=&\int_{Q_{R,\tau,T}^\eta}\left[\left(\vert\na\rhoe\vert^2-\rho^2\chi^2\right)_+\na\rhoe-\overline{w}\right](\na\rhoe-\na\rho)\dx\dt\\
&+ \int_{Q_{R,\tau,T}^\eta}\left[\overline{w}-\left(\vert\na\rho\vert^2-\rho^2\chi^2\right)_+\na\rho\right](\na\rhoe-\na\rho)\dx\dt\rightarrow0.
\end{align*}
But, notice that 
\begin{subequations}
\begin{align}
&\int_{Q_{R,\tau,T}^\eta}\left[\left(\vert\na\rhoe\vert^2-\rho^2\chi^2\right)_+\na\rhoe-\left(\vert\na\rho\vert^2-\rho^2\chi^2\right)_+\na\rho\right](\na\rhoe-\na\rho)\dx\dt\nonumber\\
=&\frac{1}{2}\int_{Q_{R,\tau,T}^\eta}\left[\left(\vert\na\rhoe\vert^2-\rho^2\chi^2\right)_++\left(\vert\na\rho\vert^2-\rho^2\chi^2\right)_+\right]\times\lvert\na(\rhoe-\rho)\rvert^2\dx\dt\label{entr.computations.1}\\
&+\frac{1}{2}\int_{Q_{R,\tau,T}^\eta}\left[\left(\vert\na\rhoe\vert^2-\rho^2\chi^2\right)_+-\left(\vert\na\rho\vert^2-\rho^2\chi^2\right)_+\right]\times\left(\lvert\na\rhoe\rvert^2-\lvert\na\rho\rvert^2\right)\dx\dt\rightarrow0.\label{entr.computations.2}
\end{align}
\end{subequations}
{Since both integrals in \eqref{entr.computations.1}, \eqref{entr.computations.2} are nonnegative,} the convergence to $0$ of \eqref{entr.computations.1} implies that we can find a subsequence (not relabeled) such that $\na\rhoe\rightarrow\na\rho$ a.e. in $Q_{R,\tau,T}^\eta \setminus\{\rho\chi\geq\vert\na\rho\vert\}$. 
W.l.o.g.~we can choose $\eta = 1/N$, $N\in\N$, and assume that $Q_{R,\tau,T}^{1/N_1}\subset Q_{R,\tau,T}^{1/N_2}$ for $N_1 < N_2$.
Since $|Q_{R,\tau,T}\setminus Q_{R,\tau,T}^\eta|<\eta$, a Cantor diagonal argument
yields the existence of a subsequence (not relabeled) such that $\na\rhoe\rightarrow\na\rho$ a.e. in $Q_{R,\tau,T} \setminus\{\rho\chi\geq\vert\na\rho\vert\}$.
Via another Cantor diagonal argument we conclude that (up to a subsequence) 
\begin{equation}\label{p.cnv.rhoe}
\na\rhoe\rightarrow\na\rho~~\mbox{ a.e. in }(0,T)\times \Rd\cap\{0< \rho\chi < \vert\na\rho\vert\}.  
\end{equation}
Thus
\begin{align}\label{flux.cnv.ae.1}
	\left(1-\frac{\chi\rhoe}{\lvert\na\rhoe\rvert}\right)_+\na\rhoe\to\left(1-\frac{\chi\rho}{\lvert\na\rho\rvert}\right)_+\na\rho\quad\mbox{a.e.~in }(0,T)\times \Rd\cap\{0< \rho\chi < \vert\na\rho\vert\}.
\end{align}
Furthermore %
\begin{align*}
	&\hspace{-6mm}\left[\left(\vert\na\rhoe\vert^2-\rho^2\chi^2\right)_+-\left(\vert\na\rho\vert^2-\rho^2\chi^2\right)_+\right]\times\left(\lvert\na\rhoe\rvert^2-\lvert\na\rhoe\rvert^2\right)\\
	=&\left[\left(\vert\na\rhoe\vert^2-\rho^2\chi^2\right)_+-\left(\vert\na\rho\vert^2-\rho^2\chi^2\right)_+\right]^2\\
	&-\left[\left(\vert\na\rhoe\vert^2-\rho^2\chi^2\right)_+-\left(\vert\na\rho\vert^2-\rho^2\chi^2\right)_+\right]\left[\left(\vert\na\rhoe\vert^2-\rho^2\chi^2\right)_--\left(\vert\na\rho\vert^2-\rho^2\chi^2\right)_-\right]\\
	=&\left[\left(\vert\na\rhoe\vert^2-\rho^2\chi^2\right)_+-\left(\vert\na\rho\vert^2-\rho^2\chi^2\right)_+\right]^2\\
	&+\left(\vert\na\rhoe\vert^2-\rho^2\chi^2\right)_+\left(\vert\na\rho\vert^2-\rho^2\chi^2\right)_-+\left(\vert\na\rhoe\vert^2-\rho^2\chi^2\right)_-\left(\vert\na\rho\vert^2-\rho^2\chi^2\right)_+.
\end{align*}
So, \eqref{entr.computations.2} gives 
\begin{align*}
	\left(\vert\na\rhoe\vert^2-\rhoe^2\chi^2\right)_+\rightarrow\left(\vert\na\rho\vert^2-\rho^2\chi^2\right)_+\quad\mbox{a.e. in }{(0,T)\times\Rd\cap \{\rho>0\}}.
\end{align*}
As a consequence we find
\begin{align*}
	\lim\limits_{\eps\rightarrow0}\left|\left(1-\frac{\chi\rhoe}{\lvert\na\rhoe\rvert}\right)_+\na\rhoe\right|\left(\vert\na\rhoe\vert+\rhoe\chi\right)_+\leq\lim\limits_{\eps\rightarrow0}\left(\vert\na\rhoe\vert^2-\rhoe^2\chi^2\right)_+=\left(\vert\na\rho\vert^2-\rho^2\chi^2\right)_+=0,\\
\mbox{a.e.~in }(0,T)\times \Rd\cap\{\rho\chi\geq\vert\na\rho\vert, ~ {\rho>0} \},
\end{align*}
hence we obtain
\begin{align}\label{flux.cnv.ae.2}
\left(1-\frac{\chi\rhoe}{\lvert\na\rhoe\rvert}\right)_+\na\rhoe\to 0=\left(1-\frac{\chi\rho}{\lvert\na\rho\rvert}\right)_+\na\rho\quad\mbox{a.e.~in }(0,T)\times \Rd\cap\{\rho\chi\geq\vert\na\rho\vert, ~ {\rho>0} \}.
\end{align}
To identify the a.e.~limit on the set $(0,T)\times \Rd\cap\{\rho=0\}$ we proceed as follows.
Consider
\begin{align*}
\left\|\left(1-\frac{\chi\rhoe}{\lvert\na\rhoe\rvert}\right)_+\na\rhoe\,{\bf 1}_{\{\rho=0\}}\right\|_{4/3}
\leq 
2\|\nabla\sqrt{\rhoe}\|_2\|\sqrt{\rhoe} {\bf 1}_{\{\rho=0\}}\|_4
=
2\|\nabla\sqrt{\rhoe}\|_2\|\rhoe {\bf 1}_{\{\rho=0\}}\|_2^{1/2}.
\end{align*}
Given that $\|\nabla\sqrt{\rhoe}\|_2$ is bounded and $\rhoe {\bf 1}_{\{\rho=0\}}\to \rho {\bf 1}_{\{\rho=0\}} = 0$
strongly in $L^2(\R^d\times(0,T))$, we conclude that 
$\left(1-\frac{\chi\rhoe}{\lvert\na\rhoe\rvert}\right)_+\na\rhoe\to 0$ strongly in $L^{4/3}(\R^d\times(0,T)\cap\{\rho=0\})$. In particular, up to subsequences
\begin{align}\label{flux.cnv.ae.0}
	\left(1-\frac{\chi\rhoe}{\lvert\na\rhoe\rvert}\right)_+\na\rhoe\rightarrow 0 = \left(1-\frac{\chi\rho}{\lvert\na\rho\rvert}\right)_+\na\rho\quad\mbox{a.e.~in } (0,T)\times \Rd\cap\{\rho=0\}.
\end{align}
Putting \eqref{flux.cnv.ae.1}--\eqref{flux.cnv.ae.0} together yields the desired conclusion
\begin{align*}
\left(1-\frac{\chi\rhoe}{\lvert\na\rhoe\rvert}\right)_+\na\rhoe\rightarrow\left(1-\frac{\chi\rho}{\lvert\na\rho\rvert}\right)_+\na\rho\quad\mbox{a.e. in } (0,T)\times \Rd,
\end{align*}
and we are finally able to identify the weak $L^2$ limit of the term $\left(1-\frac{\chi\rhoe}{\lvert\na\rhoe\rvert}\right)_+\na\rhoe$:
\begin{align}
	\left(1-\frac{\chi\rhoe}{\lvert\na\rhoe\rvert}\right)_+\na\rhoe\rightharpoonup\left(1-\frac{\chi\rho}{\lvert\na\rho\rvert}\right)_+\na\rho\quad\mbox{weakly in } L^2((0,T)\times \Rd). \label{limit_final_flux}
\end{align}

\subsection{The vanishing viscosity limit} We have now all the ingredients to pass to the limit in (\ref{smooth.eq_bis}). From \eqref{time_derivative_bound} we obtain 
\begin{align*}
	\| \partial_t\rhoe\|_{L^2(0,T;H^{-1}(\Rd))}\leq C\|\rho_{{\rm in},\eps}\|_{L^2(\Rd)}.
\end{align*}
Since $\rho_{{\rm in},\eps}$ is uniformly bounded in $L^2(0,T;\Rd)$, the last inequality yields
\begin{align*}
	\partial_t\rhoe\rightharpoonup\partial_t\rho\quad\mbox{weakly in }\;\; L^2(0,T;H^{-1}). 
\end{align*}
Convergence  (\ref{limit_final_flux}) implies that 
\begin{align*}
\int \int 	\left(1-\frac{\chi\rhoe}{\lvert\nabla\rhoe\rvert}\right)_+\nabla\rhoe \nabla \varphi \;dxdt \to \int \int  \left(1-\frac{\chi\rho}{\lvert\nabla\rho\rvert}\right)_+\nabla\rho \nabla \varphi \;dxdt.
\end{align*}
The viscous terms vanish, since 
\begin{align*}
	\lim_{\eps\rightarrow0}\left|\eps\int_0^T\int_{\Rd}\nabla\rhoe\nabla\varphi\dx\dt+\eps\int_0^T\int_{\Rd}\rhoe\varphi\dx\dt\right|&\leq\lim_{\eps\rightarrow0}\eps C\|\rho_{{\rm in},\eps}\|_{L^2(0,T;\Rd)}\|\varphi\|_{L^2(0,T; H^1(\Rd))}\\
	&=0.
\end{align*}
Hence, we have constructed a weak solution $\rho$  to \eqref{weak.formulation.1}. In the next two subsection we show that such solution is unique and converges strongly  in $L^p$ to the initial data $\rho_{in}$.  

\subsection{Obtaining the initial data} 
In what follows, we will show that $\rho(t)\rightarrow\rhoin$ strongly in $L^p(\Rd)$ as $t\to 0$.  We start by recalling that $\rhoe$ is bounded in the space
$$
\left\{
u\in L^2(0,T; H^1(\R^d))~~ : ~~ \pa_t u \in L^2(0,T; H^{-1}(\R^d))
\right\},
$$
which embeds continuously into $C^0([0,T]; L^2(\R^d))$ \cite[Prop.~23.23]{ZeidlerNonlinearIIA}. In particular it follows that $\rhoe(t)\rightharpoonup\rho(t)$ weakly in $H^{-1}(\R^d)$ for every $t\in [0,T]$, hence
\begin{equation}\label{ic.comput.1}
\lim_{\eps\to 0}\IRd \rhoe(t)\varphi dx = \IRd \rho(t)\varphi dx\quad\forall \varphi\in H^1(\R^d).
\end{equation}
Also remember that the initial condition for the approximated problem \eqref{smooth.eq_bis}
is satisfied in the sense of \eqref{in.c.ugly}.

Define for brevity $A_\eps := \eps + \left(1 - \frac{\chi\rhoe}{|\na\rhoe|} \right)_+$,
$A := \left(1 - \frac{\chi\rho}{|\na\rho|} \right)_+$.
Choose $\varphi\in H^1(\R^d)$ independent of time in \eqref{smooth.eq_bis}.
It follows
\begin{align*}
&\left| \IRd \rhoe(\tau)\varphi \dx - \IRd \rho_{{\rm in},\eps}\varphi \dx \right|\\
&\quad = \left|
\int_0^\tau\IRd A_\eps\na\rhoe\cdot\na\varphi \dx\dt 
+ \eps\int_0^\tau\IRd \rhoe\varphi \dx\dt
\right|\\
&\quad \leq (1+\eps)\left(
\|\na\rhoe\|_{L^2(0,\tau; L^2(\R^d))}\|\na\varphi\|_{L^2(0,\tau; L^2(\R^d))} +
\|\rhoe\|_{L^2(0,\tau; L^2(\R^d))}\|\varphi\|_{L^2(0,\tau; L^2(\R^d))}
\right)\\
&\quad \leq C\tau^{1/2}\|\varphi\|_{H^1(\R^d)}.
\end{align*}
Relation \eqref{ic.comput.1} allows us to take the limit $\eps\to 0$ in the above estimate and obtain
\begin{align*}
&\left| \IRd \rho(\tau)\varphi \dx - \IRd \rho_{0}\varphi \dx \right|\leq C\tau^{1/2}\|\varphi\|_{H^1(\R^d)},
\end{align*}
implying
\begin{equation}\label{ic.comput.2}
\rho(t) \to \rho_0\quad\mbox{strongly in $H^{-1}(\R^d)$ as $t\to 0$.}
\end{equation}
Define now, for $M>0$ arbitrary, 
$$
F_M(s) = \int_0^s f_M(\sigma)d\sigma,\quad 
f_M(s) = \begin{cases}
p s^{p-1} & 0\leq s\leq M\\
p(p-1)M^{p-2}(s-M) + pM^{p-1} & s > M
\end{cases}.
$$
Choose $\varphi = f_M(\rho)\chf{[t_1,t_2]}(t)\in L^2(0,T; H^1(\R^d))$ in \eqref{weak.formulation.1}.
It follows
\begin{align*}
\IRd F_M(\rho(t_2)) dx - \IRd F_M(\rho(t_1)) dx 
&= \int_{t_1}^{t_2}\IRd A f_M'(\rho)|\nabla\rho|^2 dx dt \\
&= \frac{4}{p^2}\int_{t_1}^{t_2}\IRd A f_M'(\rho)\rho^{2-p}|\nabla\rho^{p/2}|^2 dx dt.
\end{align*}
Since $f_M'(\rho)\rho^{2-p} = p(p-1)\min(\rho^{p-2},M^{p-2})\rho^{2-p}\leq p(p-1)$,
$F_M(\rho)\leq \rho^p$, Proposition \ref{Proposition LP} and the dominated convergence theorem allow us to take the limit $M\to \infty$ in the above identity and deduce
\begin{align*}
\IRd \rho(t_2)^p dx - \IRd \rho(t_1)^p dx 
&= \frac{4(p-1)}{p}\int_{t_1}^{t_2}\IRd A |\nabla\rho^{p/2}|^2 dx dt,
\end{align*}
for a.e.~$t_1, t_2\in [0,T]$. It follows that the function $t\mapsto\|\rho(t)\|_{L^p(\R^d)}^p$ belongs to $W^{1,1}(0,T)$, hence up to changing the values of $\rho : [0,T]\to L^p(\R^d)$ in a zero measure set it holds that $t\mapsto\|\rho(t)\|_{L^p(\R^d)}^p$ is continuous in $[0,T]$. As a consequence the bound $\|\rho(t)\|_{L^p(\R^d)}\leq\|\rho_{in}\|_{L^p(\R^d)}$
holds for every $t\in [0,T]$. In particular $\rho(t)$ is weakly convergent (up to subsequences) in $L^p(\R^d)$ as $t\to 0$; however, its weak limit has to be $\rho_{\in}$ due to \eqref{ic.comput.2}, hence $\rho(t)\rightharpoonup\rho_{in}$ weakly in $L^p(\R^d)$ as $t\to 0$ (no longer up to subsequences). It follows
$$
\|\rho_{in}\|_{L^p(\R^d)}\leq \liminf_{t\to 0}\|\rho(t)\|_{L^p(\R^d)}\leq 
\limsup_{t\to 0}\|\rho(t)\|_{L^p(\R^d)}\leq \|\rho_{in}\|_{L^p(\R^d)} ,
$$
implying that $\|\rho(t)\|_{L^p(\R^d)}\to \|\rho_{in}\|_{L^p(\R^d)}$ as $t\to 0$. This fact and the weak convergence $\rho(t)\rightharpoonup\rho_{in}$ in $L^p(\R^d)$, imply that 
 $\rho(t)\rightarrow\rhoin$ strongly in $L^p(\Rd)$ as $t\to 0$.

\subsection{Uniqueness} Let $\rho_1$ and $\rho_2$ be two solutions to \eqref{weak.formulation.1} with the same initial data $\rho_{in}$. As in Section \ref{sec:entropy_limit}, we subtract equation \eqref{weak.formulation.1} for $\rho_2$ tested with $\varphi_2=\frac{\rho_1+\sigma}{\rho_2+\sigma}{-1} =\frac{ \rho_1 - \rho_2}{\rho_2 + \sigma}$ to the equation for $\rho_1$ tested with $\varphi_1=\log\left(\frac{\rho_1+\sigma}{\rho_2+\sigma}\right)$. Following the same computations as in \eqref{equiv_entr}, we get 
\begin{align*}%
	\int_0^T\int_{\Rd}&\left[ \varphi_1 \pa_t\rho_1 -\varphi_2 \pa_t\rho_2 \right]\dx\;dt=\nonumber\\
	&=\int_0^T \frac{d}{dt}\int_{\Rd} \left({\rho_1+\sigma}\right)\log{\left(\frac{\rho_1+\sigma}{\rho_2+\sigma}\right)}-\rho_1+\rho_2  \dx \dt = \int_0^T \frac{d}{dt} H[\rho_1+\sigma | \rho_2+\sigma]\;dt. 
\end{align*}
On the other hand, substituting $\partial_t \rho_1$ and $\partial_t \rho_2$ with the respective right hand side of \eqref{weak.formulation.1}, after some integration by parts we obtain:
\small 
\begin{align}\label{uniqueness comp}
	&\int_0^T\int_{\Rd}\left[ \varphi_1 \pa_t\rho_1 -\varphi_2 \pa_t\rho_2 \right]\dx\;dt \\
	&=-\frac{1}{2}\int_0^T\hspace{-2mm}\IRd(\rho_1+\sigma)\left[\left(1-\frac{\chi(\rho_1+\sigma)}{\lvert\na\rho_1\rvert}\right)_+\hspace{-2mm}-\left(1-\frac{\chi(\rho_2+\sigma)}{\lvert\na\rho_2\rvert}\right)_+\right] \cdot  \nonumber \\
 & \;\;\;\;\;\;\;\;\;\;  \cdot \left(\lvert\na\log(\rho_1+\sigma)\rvert^2-\lvert\na\log(\rho_2+\sigma)\rvert^2\right)\dx\dt\nonumber\\
	&\;\;\;-\frac{1}{2}\int_0^T\hspace{-2mm}\IRd(\rho_1+\sigma)\left\vert\na\log\left(\frac{\rho_1+\sigma}{\rho_2+\sigma}\right)\right\vert^2\left[\left(1-\frac{\chi(\rho_1+\sigma)}{\lvert\na\rho_1\rvert}\right)_++\left(1-\frac{\chi(\rho_2+\sigma)}{\lvert\na\rho_2\rvert}\right)_+\right]\dx\dt \nonumber \\
 &  \;\;\;+ \xi_1+\xi_2,\nonumber
\end{align}
\normalsize
where
\begin{align*}
	\xi_1:=-\int_0^T\hspace{-2mm}\IRd\left[\left(1-\frac{\chi\rho_1}{\lvert\na\rho_1\rvert}\right)_+-\left(1-\frac{\chi(\rho_1+\sigma)}{\lvert\na\rho_1\rvert}\right)_+\right]\na\rho_1 \na\left(\log\frac{\rho_1+\sigma}{\rho_2+\sigma}\right)\dx\dt,
\end{align*}
and 
\begin{align*}
	\xi_2:=\int_0^T\hspace{-2mm}\IRd\left[\left(1-\frac{\chi\rho_2}{\lvert\na\rho_2\rvert}\right)_+-\left(1-\frac{\chi(\rho_2+\sigma)}{\lvert\na\rho_2\rvert}\right)_+\right]\na\rho_2\na\left(\frac{\rho_1+\sigma}{\rho_2+\sigma}\right)\dx\dt.
\end{align*}
The first two integrals in (\ref{uniqueness comp}) are positive, the first because of monotonicity and the second because it is just the product of positive functions. Since both have negative signs in front, we can ignore them and we get 
$$
\int_0^T\int_{\Rd}\left[ \varphi_1 \pa_t\rho_1 -\varphi_2 \pa_t\rho_2 \right]\dx\;dt \le  \xi_1 + \xi_2.
$$
{We want to show that $\limsup_{\sigma\to 0}(\xi_1 + \xi_2)\leq 0$. Dominated convergence theorem will be used. We only have to show that the integrands in $\xi_1$ and $\xi_2$ are dominated by an integrable function (uniformly in $\sigma$). We start with $\xi_1$:}
\begin{align*}
	\xi_1\leq&\int_0^T\hspace{-2mm}\IRd\left[\left(1-\frac{\chi\rho_1}{\lvert\na\rho_1\rvert}\right)_+-\left(1-\frac{\chi(\rho_1+\sigma)}{\lvert\na\rho_1\rvert}\right)_+\right]\na\rho_1\frac{\na\rho_2}{\rho_2+\sigma}\dx\dt\\
	\leq&\underbrace{\frac{1}{2}\int_0^T\hspace{-2mm}\IRd\left[\left(1-\frac{\chi\rho_1}{\lvert\na\rho_1\rvert}\right)_+-\left(1-\frac{\chi(\rho_1+\sigma)}{\lvert\na\rho_1\rvert}\right)_+\right]\frac{|\na\rho_1|^2}{\rho_2+\sigma}\dx\dt}_{\text{(I)}}\\
	&+\underbrace{\frac{1}{2}\int_0^T\hspace{-2mm}\IRd\left[\left(1-\frac{\chi\rho_1}{\lvert\na\rho_1\rvert}\right)_+-\left(1-\frac{\chi(\rho_1+\sigma)}{\lvert\na\rho_1\rvert}\right)_+\right]\frac{|\na\rho_2|^2}{\rho_2+\sigma}\dx\dt}_{\text{(II)}}.
\end{align*}
It is easy to see that {the integrand in} (II) is controlled by $|\na\sqrt\rho_2|^2\in L^1(\R^d\times (0,T))$. For (I) we use the Lipschitz condition \eqref{lipschitz condition} and we get
\begin{align*}
|\mbox{Integrand in (I)}|\leq\frac{1}{2}|\na\rho_1| \leq 
 |\na\sqrt\rho_1| \sqrt{\rho_1}\in L^1(\R^d\times (0,T)).
\end{align*}
This shows that $\xi_1$ is uniformly bounded with respect to $\sigma$. Let us estimate $\xi_2$:
\begin{align*}
	\xi_2\leq&\int_0^T\hspace{-2mm}\IRd\left[\left(1-\frac{\chi\rho_2}{\lvert\na\rho_2\rvert}\right)_+-\left(1-\frac{\chi(\rho_2+\sigma)}{\lvert\na\rho_2\rvert}\right)_+\right]\frac{\na\rho_2\na\rho_1}{\rho_2+\sigma}\dx\dt\\
	&-\int_0^T\hspace{-2mm}\IRd\left[\left(1-\frac{\chi\rho_2}{\lvert\na\rho_2\rvert}\right)_+-\left(1-\frac{\chi(\rho_2+\sigma)}{\lvert\na\rho_2\rvert}\right)_+\right]\frac{(\rho_1+\sigma)|\na\rho_2|^2}{(\rho_2+\sigma)^2}\dx\dt\\
	\leq&\;\frac{1}{2}\int_0^T\hspace{-2mm}\IRd\left[\left(1-\frac{\chi\rho_2}{\lvert\na\rho_2\rvert}\right)_+-\left(1-\frac{\chi(\rho_2+\sigma)}{\lvert\na\rho_2\rvert}\right)_+\right]\frac{|\na\rho_1|^2}{\rho_1+\sigma}\dx\dt\\
	&-\frac{1}{2}\int_0^T\hspace{-2mm}\IRd\left[\left(1-\frac{\chi\rho_2}{\lvert\na\rho_2\rvert}\right)_+-\left(1-\frac{\chi(\rho_2+\sigma)}{\lvert\na\rho_2\rvert}\right)_+\right]\frac{(\rho_1+\sigma)|\na\rho_2|^2}{(\rho_2+\sigma)^2}\dx\dt\\
	\leq& \;{\frac{1}{2}\int_0^T\hspace{-2mm}\IRd\left[\left(1-\frac{\chi\rho_2}{\lvert\na\rho_2\rvert}\right)_+-\left(1-\frac{\chi(\rho_2+\sigma)}{\lvert\na\rho_2\rvert}\right)_+\right]\frac{|\na\rho_1|^2}{\rho_1}\dx\dt.}
\end{align*}
{The integrand in the last expression is controlled by $|\nabla\sqrt{\rho_1}|\in L^1(\R^d\times (0,T))$. Therefore by dominated convergence we can take }
the limit $\sigma\rightarrow0$ in \eqref{uniqueness comp} to obtain
\begin{align*}
\int_{\Rd} & \left[  \rho_1 \log{\left(\frac{\rho_1}{\rho_2}\right)}- \rho_1+ \rho_2\right](T)  \dx 
 \; {\leq \limsup_{\sigma\to 0}(\xi_1 + \xi_2)}  \; \leq \;0.
\end{align*}
This implies that $\rho_1=\rho_2$ and finishes the proof of Theorem \ref{thr.ex}.

\end{document}